\documentclass[11pt,english]{article}
\usepackage[margin=1in]{geometry}
\usepackage{graphicx}
\usepackage{todonotes}
\usepackage{amsmath,amssymb,amsthm}
\usepackage{mathtools}
\usepackage{url}

\usepackage{thmtools}
\usepackage{thm-restate}

\usepackage[capitalize]{cleveref}

\usepackage{enumerate}

\setuptodonotes{inline}

\newtheorem{theorem}{Theorem}
\newtheorem{proposition}[theorem]{Proposition}
\newtheorem{lemma}[theorem]{Lemma}

\newtheorem{definition}[theorem]{Definition}
\newtheorem{claim}{Claim}

\newtheorem*{theorem*}{Theorem}

\theoremstyle{remark}
\newtheorem*{proofofclaim}{Proof of claim}
\newcommand*{\QEDclaim}{\hfill\ensuremath{\diamondsuit}}

\DeclareMathOperator{\eg}{eg} 
\DeclareMathOperator{\fw}{fw} 
\DeclareMathOperator{\intt}{int} 
\newcommand{\surf}{\Sigma}
\newcommand{\CC}{\mathcal{C}}
\newcommand{\DD}{\mathcal{D}}

\newcommand{\Z}{\mathbb{Z}}
\newcommand{\N}{\mathbb{N}}

\newcommand{\R}{\mathbb{R}}

\newcommand{\contract}{/}
\newcommand{\delete}{\setminus}

\newcommand{\setdiff}{-}
\newcommand{\numxcaps}{\gamma}

\let\ge\relax
\newcommand{\ge}{\geqslant}
\let\le\relax
\newcommand{\le}{\leqslant}

\title{Face covers and rooted minors in bounded genus graphs}
\author{Samuel Fiorini \and Stefan Kober \and Michał T.\ Seweryn \and Abhinav Shantanam \and Yelena Yuditsky}
\date{\today}

\begin{document}

\maketitle

\begin{abstract}
A {\em rooted graph} is a graph together with a designated vertex subset, called the {\em roots}. In this paper, we consider rooted graphs embedded in a fixed surface. A collection of faces of the embedding is a {\em face cover} if every root is incident to some face in the collection. We prove that every $3$-connected, rooted graph that has no rooted $K_{2,t}$ minor and is embedded in a surface of Euler genus $g$, has a face cover whose size is upper-bounded by some function of $g$ and $t$, provided that the face-width of the embedding is large enough in terms of $g$. In the planar case, we prove an unconditional $O(t^4)$ upper bound, improving a result of B\"ohme and Mohar~\cite{BM02}. The higher genus case was claimed without a proof by B\"ohme, Kawarabayashi, Maharry and Mohar~\cite{BKMM08}.
\end{abstract}

\section{Introduction}

Let \(G\) be a plane graph, and let \(R \subseteq V(G)\) be a distinguished set of \emph{root} vertices.
A \emph{face cover} for the rooted graph \((G, R)\) is a collection of faces of \(G\) such that every root \(r \in R\) is incident with at least one of the faces in this collection.

When does $(G,R)$ admit a small face cover? One structure preventing the existence of a small face cover is a large \emph{rooted \(K_{2, t}\)-model}, that is a sequence of pairwise disjoint connected subgraphs \((X_1, X_2; Y_1, \ldots, Y_t)\) of \(G\) such that each \(X_i\) is adjacent to each \(Y_j\) in \(G\), and each \(Y_j\) contains at least one vertex from the root set \(R\). (The sets \(X_1\) and \(X_2\) may or may not contain a root.)
If there exists a rooted \(K_{2, t}\)-model \((X_1, X_2; Y_1, \ldots, Y_t)\), then every face of the embedding of \(G\) is incident with roots in at most two sets \(Y_j\), and thus, every face cover contains at least \(\lceil t/2 \rceil\) faces. We say that a rooted graph contains a {\em rooted $K_{2,t}$ minor} if it contains a rooted $K_{2,t}$-model.\footnote{Alternatively, one can readily generalize the usual notion of a graph minor to rooted graphs in terms of graph operations. The main differences are that when an edge is contracted, we add the resulting vertex in the root set iff at least one of its endpoints is a root, and that we allow a new operation that removes one vertex from the root set.}
Note that, unless explicitly mentioned, any notion for a rooted graph has the same definition as for a non-rooted graph. 

B\"ohme and Mohar~\cite{BM02} showed that if \(G\) is \(3\)-connected, then a rooted \(K_{2, t}\)-model for some large value of \(t\) is in some sense the only obstruction preventing a small face cover. More precisely, they showed that for every integer \(t \ge 1\), there exists an integer \(k = k(t) \ge 1\) such that if there does not exist a rooted \(K_{2, t}\)-model, then there exists a face cover of size at most \(k\).
In the original proof, \(k\) depends exponentially on \(t\). Our first result is a significantly simpler proof with \(k \in O(t^4)\).

\begin{theorem} \label{thm:planar}
There exists a function $f_{\ref{thm:planar}} : \Z_{\ge 1} \to \Z_{\ge 0}$ with $f_{\ref{thm:planar}}(t) = O(t^4)$ such that the following holds. If $(G,R)$ is a $3$-connected plane rooted graph without a rooted $K_{2,t}$ minor, then $(G,R)$ admits a face cover of size at most $f_{\ref{thm:planar}}(t)$.%
\end{theorem}

We do not know whether our bound of $f_{\ref{thm:planar}}(t) = O(t^4)$ is tight. However, we observe that $f_{\ref{thm:planar}}(t) = \Omega(t^2)$ for any valid upper bound $f_{\ref{thm:planar}}(t)$ on the size of a face cover, see \Cref{fig:windmill} and \Cref{prop:windmill} below.

\begin{figure}[ht]
\centering
\includegraphics[width=0.3\textwidth]{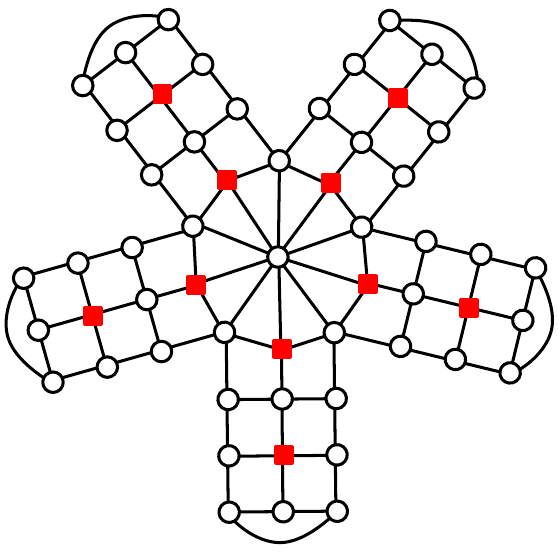}
\caption{The {\em windmill graph} of parameter $t = 10$. Generalizing this construction gives a family of $3$-connected, plane rooted graphs with no rooted $K_{2,t}$ minor in which every face cover has size $\Omega(t^2)$. See \Cref{prop:windmill} below for a proof.}
\label{fig:windmill}
\end{figure}

Our second result covers the higher genus case (Euler genus $g \ge 1$). It was claimed as a remark to Theorem 3.4 in B\"ohme, Kawarabayashi, Maharry and Mohar~\cite{BKMM08}, without proof. For background information related to embeddings of graphs on surfaces, including the face-width of an embedding, the genus of a graph, and more, see Section \ref{sec:background}.

\begin{theorem}\label{thm:bounded_genus}
There are functions $w_{\ref{thm:bounded_genus}}:\Z_{\ge1}\to\Z_{\ge0}$ and $f_{\ref{thm:bounded_genus}}: \mathbb{Z}_{\ge 1} \times \mathbb{Z}_{\ge 1} \to \mathbb{Z}_{\ge 0}$, such that the following holds. If $(G, R)$ is a $3$-connected rooted graph without a rooted \(K_{2,t}\) minor, embedded in a surface of Euler genus $g$ with face-width at least $w_{\ref{thm:bounded_genus}}(g)$, then $(G,R)$ admits a face cover of size at most $f_{\ref{thm:bounded_genus}}(g, t)$.
\end{theorem}

The function $f_{\ref{thm:bounded_genus}}(g,t)$ resulting from our proof is $O(g f_{\ref{thm:planar}}(t)) = O(g \cdot t^4)$. We do not know whether it is optimal. However, we point out that some lower bound on the face-width is necessary, for instance because of the graphs described in~\Cref{fig:bagel}.

\begin{figure}[hb]
\centering
\includegraphics[width=0.35\textwidth]{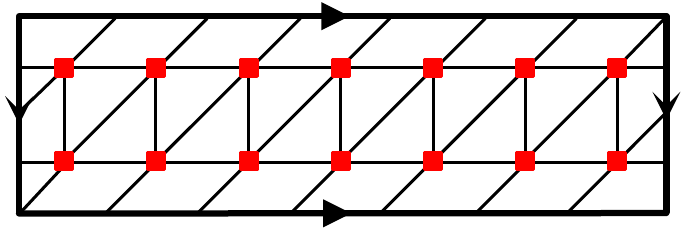}
\caption{The figure shows a {\em bagel graph} embedded in the torus, with $n = 14$ vertices, all included in the root set. The construction generalizes to every even $n \ge 6$ and gives a $4$-connected, $n$-vertex rooted graph $B_n$ embedded on the torus and without rooted $K_{2,t}$ minor for $t \ge 5$. However every face cover has size $\Omega(n)$. Notice that the face-width of $B_n$ is $2$. See \Cref{prop:bagel} for details.}
\label{fig:bagel}
\end{figure}

To conclude this introduction, we give some motivation for \Cref{thm:planar,thm:bounded_genus}. Our understanding is that the earliest motivation lies in the study of apex graphs and their (Euler) genus. Let $G$ denote a planar graph, and let $H$ denote any graph obtained from $G$ by adding a single vertex, say $v_0$, adjacent to any subset of vertices of $G$.
Then $H$ is known as an {\em apex-planar graph}. Although removing a single vertex from $H$ results in a planar graph, it is well-known that $\eg(H)$, the Euler genus of $H$, cannot be bounded by a constant.

Letting $R := N_H(v_0)$ denote the neighborhood of $v_0$ in $H$, consider the rooted graph $(G,R)$. Suppose that $(G,R)$ has a rooted $K_{2,t}$ minor, then $H$ has a $K_{3,t}$ minor. Since $\eg(K_{3,t}) = \Omega(t)$ and since the genus is minor-monotone, we conclude that $\eg(H) \ge \eg(K_{3,t}) = \Omega(t)$. Hence, having a rooted $K_{2,t}$ minor in $(G,R)$ with $t$ large is an obstruction to having small genus for the corresponding apex graph $H$. 

Now suppose that $(G,R)$ has $\tau$ faces $f_1$, \ldots, $f_\tau$ forming a minimal face cover. Let $R_1,\dots,R_\tau$ denote a partition of the root set $R$ into nonempty subsets, where each $R_i$ consists of roots incident to face $f_i$. Consider the graph $H'$ obtained by splitting the vertex $v_0$ in $\tau$ vertices $v_{0,1},\dots,v_{0,\tau}$, making each $v_{0,i}$ adjacent to all vertices in $R_i$. Consider also the graph $H''$ obtained from $H'$ by adding an arbitrary spanning tree on the vertices $v_{0,1}$, \ldots, $v_{0,\tau}$. Notice that (i) $H'$ is planar, since we can draw each $v_{0,i}$ inside $f_i$, the corresponding face of $G$, (ii) $H''$ can be obtained from $H'$ by adding $\tau - 1$ edges, and (iii) $H$ is a minor of $H''$. Since the addition of a single edge to a graph increases its genus by at most $2$, we conclude $\eg(H) \le \eg(H'') \le \eg(H') + 2 (\tau-1) = 2 \tau - 2$.

In this context, \Cref{thm:planar} states that (provided that $G$ is $3$-connected) either $\eg(H) \le O(t^4)$ or $H$ contains a $K_{3,t}$ minor, certifying that $\eg(H) \ge \Omega(t)$. \Cref{thm:bounded_genus} can be seen as a generalization of this result to graphs $H$ obtained from a bounded genus graph $G$ by adding a single vertex. 

A related work is that of Kawarabayashi and Sidiropoulos~\cite{KS15}, who show that $\eg(H) = O(\tau^2)$, where $H$ is an apex-planar graph with apex $v_0$ and $\tau$ is the minimum size of a face cover of $G = H - v_0$. 

A second motivation for our main theorems lies in the recent work of Aprile, Fiorini, Joret, Kober, Seweryn, Weltge and Yuditsky~\cite{AFJKSWY25}. In that paper, it is shown that certain integer programs with bounded subdeterminants can be reduced to integer programs of the form
\begin{equation}
\label{eq:IP}
\max \left\{ p^\intercal y : \ell(v,w) \le y(v) - y(w) \le u(v,w) \text{ for all } (v,w) \in E(G), \, Wy = d, \, y \in \Z^{V(G)} \right\},
\end{equation}
where $k$ is a constant, $G$ is a directed graph, $p \in \Z^{V(G)}$, $\ell, u \in \Z^{E(G)}$, $W \in \Z^{[k] \times V(G)}$ and $d \in \Z^{k}$. They define $R$ as the set of vertices $v$ such that $W(\cdot,v) \neq \mathbf{0}$. This includes a vertex $v$ in the root set whenever it appears with a nonzero coefficient in some of the $k$ constraints $Wx = d$. 

Aprile \emph{et al.}~\cite{AFJKSWY25} observe that the rooted graph $(G,R)$ has no rooted $K_{2,t}$ minor, for some constant $t$ (larger than $k$). Furthermore, they prove the following variant of the graph minor structure theorem: every graph with no rooted $K_{2,t}$ minor admits a certain decomposition, in which bounded genus graphs are a main building block.

Essentially, this reduces \eqref{eq:IP} to the case where $(G,R)$ is a $3$-connected, rooted graph that has no rooted $K_{2,t}$ minor and is embedded in a bounded genus surface with large face-width. In this case, \Cref{thm:planar,thm:bounded_genus} turn out to be the key to solving \eqref{eq:IP} in the sense that they allow the efficient guessing of all the variables $y(v)$ where $v \in R$. Once these variables are fixed, the resulting integer program can be solved in (strongly) polynomial time, since its constraint matrix is totally unimodular\footnote{A matrix is {\em totally unimodular} if all of its subdeterminants are in $\{-1,0,1\}$.}.  

We conclude this introduction with a short outline of the paper. \Cref{sec:background} provides background on graph terminology, Schynder embeddings and graphs on surfaces. The proof of \Cref{thm:planar} is given in \Cref{sec:planar}, and that of \Cref{thm:bounded_genus} in \Cref{sec:bounded_genus}.

\section{Background} \label{sec:background}

In this section, we provide the necessary background for the proofs of \Cref{thm:planar,thm:bounded_genus}. 

\subsection{Graph terminology} \label{sec:background_graphs}

We generally follow the textbook~\cite{Diestel}. 

Let $G = (V(G),E(G))$ denote a graph.\footnote{All graphs in this paper are undirected, and most are simple. For now, we assume that $G$ is a simple undirected graph.}
For $X \subseteq V(G)$, we let $G - X$ denote the graph obtained from $G$ by deleting all the vertices in $X$ (as well as all the edges incident to some vertex of $X$), we let $N(X) := \{w \in V(G - X) \mid \exists v \in X : vw \in E(G)\}$ denote the \emph{open neighborhood of $X$}, and we let $N[X] := X \cup N(X)$ denote the \emph{closed neighborhood of $X$}.
The \emph{cut} $\delta(X)$ is defined as the set of edges with exactly one end in $X$, that is, $\delta(X) := \{vw \in E(G) \mid v \in X,\ w \notin X\}$.
If $v \in V(G)$ is a vertex of $G$, then we write $G-v$ instead of $G-\{v\}$ and we let $N(v) := N(\{v\})$, $N[v] := N[\{v\}]$ and $\delta(v) := \delta(\{v\})$.

If $e \in E(G)$ is an edge of $G$, we denote by $G \delete e$ the graph obtained by \emph{deleting} $e$ and $G \contract e$ the graph obtained by \emph{contracting} $e$.

Given a graph $G$ and a distinguished vertex $r\in V(G)$, an \textit{$r$-in-arborescence} is a spanning tree $T$ of $G$, where we assign directions to the edges of $T$, such that every vertex in $V(G)$ has a directed path to $r$ in $T$. A vertex $v$ is the {\em parent} of the vertex $u$ in $T$, if $T$ contains the directed arc $(u,v)$. A vertex $u$ is a {\em descendant} of vertex $v$ in $T$, if there is a directed path from $u$ to $v$ in $T$. In particular, vertex $v$ is a descendant of itself. Finally, vertex $v$ is an {\em ancestor} or vertex $u$ if $u$ is a descendant of $v$.

\subsection{Schnyder embeddings} \label{sec:background_Schnyder}

Our proof of \Cref{thm:planar} uses (generalized) Schnyder embeddings, as defined by Felsner~\cite{Fel01}. Below, we extract a minimal set of notions and results from that paper in preparation for our proof.

Let $G$ be any $3$-connected planar graph, and let $a_1$, $a_2$ and $a_3$ denote distinct vertices, all on the same facial cycle (since $G$ is $3$-connected and planar, its facial cycles do not depend on the embedding).
Let $\Delta := \{x \in \R^3_{\ge 0} : x_1 + x_2 + x_3 = 1\}$ denote the standard $2$-dimensional simplex in $\R^3$, and let $e_1 := (1,0,0)$, $e_2 := (0,1,0)$ and $e_3 := (0,0,1)$ denote the vertices of $\Delta$. For $z \in \Delta$ and $i \in [3]$, let $\Pi_i(z) := \{x \in \Delta : x_{i-1} \le z_{i-1},\ x_{i+1} \le z_{i+1}\}$ denote the \emph{parallelogram with corners $z$ and $e_i$} (here and throughout the paper, indices are computed cyclically).
See Figure~11 in \cite{Fel01} or \Cref{fig:BM1} below for an illustration.

\begin{theorem} \label{thm:Schnyder_embedding}
Under the above assumptions, $G$ has a straight-line, planar embedding $\mu : V(G) \to \Delta$ and oriented spanning trees $T_1$, $T_2$ and $T_3$ such that, for each $i \in [3]$,

\begin{enumerate}[(i)]
\item $T_i$ is an $a_i$-in-arborescence (that is, all the directed edges of $T_i$ point toward $a_i$);
\item $\mu(a_i) = e_i$;
\item for each $v \in V(G - a_i)$, $\Pi_i(\mu(v))$ contains a unique point from $\{\mu(w) : w \in N(v)\}$, namely, $\mu(v')$ where $v'$ is the parent of $v$ in $T_i$. 
\end{enumerate}
\end{theorem}

Schnyder embeddings are discussed in Section 2.4 of~\cite{Fel01}. As a matter of fact, the embeddings constructed by Felsner satisfy extra properties, including the following. Letting $f$ denote the number of faces of $G$, there exists an embedding $\mu : V(G) \to \Delta$ as in \Cref{thm:Schnyder_embedding} that is \emph{convex} in the sense that every face distinct from the outer face is the set of interior points of a convex polygon, and such that $\mu(v) \in \frac{1}{f-1} \Z^3$ for each $v \in V(G)$. This constitutes a generalization to all $3$-connected planar graphs of a result of Schnyder on $3$-connected planar triangulations~\cite{Sch89,Sch90}. 

We point out that for general $3$-connected planar graphs, oriented spanning trees $T_1$, $T_2$ and $T_3$ may have common edges (with some restrictions, see~\cite{Fel01}). However, if $G$ is a triangulation they can be chosen to be edge-disjoint. 

\subsection{Graphs on surfaces} \label{sec:background_surfaces}

This section establishes some background notions needed for~\Cref{thm:bounded_genus}. See the book by Armstrong~\cite{Armstrong} for an introduction to topology, or that by Mohar and Thomassen~\cite{MT01} for a thorough treatment of graphs on surfaces.

A \emph{surface} $\surf$ is a non-empty compact connected Hausdorff topological space in which every point has a neighborhood that is homeomorphic to the plane. We will encounter slightly more general topological spaces, which we define right away. A \emph{surface (possibly) with boundary} is a topological space \(\Pi\) that can be obtained from a surface \(\surf\) by deleting the interiors of pairwise disjoint closed disks \(\Delta_1, \ldots, \Delta_c \subseteq \surf\) (we allow $c = 0$). The boundaries of the disks \(\Delta_1, \ldots, \Delta_c\) are the \emph{cuffs} of \(\Pi\).

Some examples of surfaces with boundary are: a closed disk (obtained from a sphere by removing the interior of a single closed disk), a M\"obius band (obtained from a projective plane by removing the interior of a single closed disk), and a cylinder (obtained from a sphere by removing the interiors of two disjoint closed disks).

The operation of gluing a closed disk along its boundary to a cuff of a surface with boundary is called \emph{capping} the cuff. Capping all cuffs of a surface with boundary yields a surface. The surface obtained from \(\Pi\) by capping all the cuffs is denoted by \(\widehat{\Pi}\). 

For $h, \numxcaps, c \in \mathbb{N}$, let $\Pi(h, \numxcaps, c)$ denote the surface with boundary obtained from a sphere by removing the interiors of $2h+\numxcaps+c$ pairwise disjoint closed disks and gluing $h$ cylinders and $\numxcaps$ M\"obius strips along any $2h +\numxcaps$ cuffs (each cylinder is glued along two cuffs and each M\"obius band along one), leaving the remaining $c$ cuffs untouched. It is known that every surface with boundary is homeomorphic to $\Pi(h, \numxcaps, c)$ for some $h, \numxcaps, c \in \mathbb{N}$.\footnote{We remark that $h$ and $\numxcaps$ are however not uniquely determined, since $\Pi(h, \numxcaps, c)$ is homeomorphic to $\Pi(h+1, \numxcaps-2, c)$ whenever $\numxcaps \ge 3$.} 

Consider a surface with boundary $\Pi$. 
A \emph{simple curve} in $\Pi$ is the image of a continuous function $f:[0,1]\rightarrow \Pi$ satisfying $f(x) \neq f(y)$ whenever $x \neq y$ and $|x-y| \neq 1$.
The simple curve $e=f([0,1])$ \emph{connects} the endpoints $f(0)$ and $f(1)$ (possibly, $f(0) = f(1)$). Its \emph{interior} is defined as $e \setdiff \{f(0),f(1)\} = f((0,1))$. If $f(0) \neq f(1)$, we call $e$ a \emph{simple open curve} or \emph{arc}. In case $f(0) = f(1)$, we call $e$ a \emph{simple closed curve} or \emph{loop}.

A loop contained in the interior of $\Pi$ is called \emph{two-sided} if it has a closed neighborhood that is homeomorphic to a cylinder and \emph{one-sided} otherwise. A surface with boundary $\Pi$ is \emph{orientable} if all its loops are two-sided. It is known that $\Pi(h, \numxcaps, c)$ is orientable if and only if $\numxcaps = 0$.

We say that a graph $G$ is \emph{embedded} in a surface with boundary $\Pi$ if the vertices of $G$ are distinct points in $\Pi$ and every edge of $G$ is a simple curve in $\Pi$ that connects the endpoints of the edge, such that its interior is disjoint from other vertices and edges. We say that a surface with boundary \(\Pi\) is \emph{contoured} by a graph \(G\) embedded in it if each cuff of \(\Pi\) coincides with a cycle of \(G\). For simplicity, we often identify an embedded graph with the corresponding subset of the surface, when no confusion can occur.

Consider a graph $G$ embedded in a surface with boundary $\Pi$.
The {\em faces} of the embedded graph $G$ are the (arcwise) connected components of $\Pi \setdiff G$.   
The embedding of $G$ in $\Pi$ is said to be \emph{cellular} (or \emph{$2$-cell}) if each one of its faces is an open disk. (If $\Pi$ has cuffs, this implies that each cuff coincides with a cycle of $G$, that is, $\Pi$ is contoured by $G$.)

Now assume that $G$ is embedded in a surface $\surf$.
The dual graph $G^*$ has the set of faces of $G$ as vertex set, and its edges bijectively correspond to the edges of $G$. The edge $e^* \in E(G^*)$ corresponding to an edge $e \in E(G)$ connects the face(s) of $G$ incident to $e$ in the embedding of $G$. The dual graph $G^*$ also admits an embedding in $\surf$.\footnote{We point out that the dual of a simple graph is not necessarily simple: loops and parallel edges can appear.}

Next, assume that the embedding of $G$ in $\surf$ is cellular. Letting $F(G)$ denote the set of faces of $G$, it is known that the quantity $2+|E(G)|-|V(G)|-|F(G)|$ is an invariant, known as the \emph{Euler genus} of surface $\surf$. The \emph{Euler genus} of a surface with boundary $\Pi$ is defined as the Euler genus of its corresponding surface with capped cuffs $\widehat{\Pi}$. We denote the Euler genus of $\Pi$ by $\eg(\Pi)$.

The \emph{Euler genus} of a graph $G$, denoted by $\eg(G)$, is defined as the minimum Euler genus of a surface in which $G$ can be embedded. 

It is known that the Euler genus of $\Pi(h, \numxcaps, c)$ is given by $\eg(\Pi(h, \numxcaps, c)) = 2h+\numxcaps$. Furthermore, Euler genus, orientability and number of cuffs form a complete set of invariants for surfaces with boundary: two surfaces with boundary $\Pi$ and $\Pi'$ are homeomorphic if and only if they have the same Euler genus, orientability and number of cuffs.

The following corollary will come in handy later in \Cref{sec:bounded_genus}.

\begin{proposition} \label{prop:gluing_surfaces_with_bd}
For $i \in [2]$, let $\Pi_i$ denote a surface with boundary with $c_i \geq 1$ cuffs. Consider the surface with boundary $\Pi$ obtained by gluing $\Pi_1$ and $\Pi_2$ along any pair of cuffs, one in $\Pi_1$ and the other in $\Pi_2$. Then, we have $\eg(\Pi) = \eg(\Pi_1) + \eg(\Pi_2)$, $\Pi$ is orientable if and only if both $\Pi_1$ and $\Pi_2$ are, and $\Pi$ has $c_1 + c_2 - 2$ cuffs. This determines $\Pi$ up to homeomorphism.
\end{proposition}

We resume the discussion of the graph $G$ embedded in the surface $\surf$. 

A graph $H$ is called a \emph{surface minor} of $G$ if $H$ is a minor of $G$ embedded in $\surf$ in such a way that the embedding of $H$ can be obtained from that of $G$ by contracting edges, and deleting edges and isolated vertices directly in $\surf$, and up to homemorphism. Whenever we contract an edge \(e\), we adjust the embedding by modifying it in a small neighbourhood of \(e\) disjoint from the vertices and edges non-incident with \(e\).

A simple curve $\ell$ in $\surf$ is said to be {\em $G$-normal} if 
it intersects $G$ only in vertices. 
A {\em noose} is a simple, closed, $G$-normal and non-contractible curve in $\surf$. For a surface $\surf$ that is not a sphere, the {\em face-width} (also known as {\em representativity}) of the embedding of $G$ in $\surf$, denoted by $\fw(G)$, is the minimum of $|\ell \cap V(G)|$ over all nooses $\ell$. The face-width of any graph embedded on the sphere is set to be infinity. 

Notice that cellular embeddings coincide with embeddings of face-width at least $1$. An embedding is said to be \emph{polyhedral} if it is cellular, each face is bounded by a cycle, and every two facial cycles are either disjoint, or have one vertex in common, or one edge in common. The following result characterizes polyhedral embeddings.

\begin{proposition}[{Mohar~\cite[Propositions 3.8 and 3.9]{Moh97}}]\label{faces-cycles}
    Let $G$ be a graph embedded in a surface $\surf$.
    Then, all faces of $G$ are open disks bounded by a cycle if and only if $G$ is $2$-connected and has face-width at least $2$.
    Further, $G$ is polyhedrally embedded if and only if $G$ is $3$-connected and has face-width at least $3$.
\end{proposition}

Let \(G\) be a graph embedded in a surface with boundary \(\Pi\). Let \(C\) denote a cycle of \(G\) disjoint from the cuffs of \(\Pi\). \emph{Cutting} along cycle \(C\) is a standard operation that changes both \(\Pi\) and \(G\), as follows (see for instance \cite{CdV21}).

In order to cut \(\Pi\) along \(C\), we remove \(C\) from \(\Pi\) and compactify the resulting space in the natural way, replacing each point \(x \in C\) by two points \(x'\) and \(x''\), see \Cref{fig:cutting_along_C} for an illustration.
This yields a new topological space \(\Pi'\). In case \(C\) is one-sided, \(\Pi'\) is a surface with boundary that has one more cuff than \(\Pi\). In case \(C\) is two-sided, \(\Pi'\) is either a surface with boundary or the disjoint union of two surfaces with boundary. In either case, the (total) number of cuffs of \(\Pi'\) is the number of cuffs of \(\Pi\) plus $2$.
 
We perform a similar operation to \(G\) in order to obtain a graph \(G'\) embedded in \(\Pi'\). 
Each vertex of \(G - V(C)\) has one corresponding vertex in \(G'\), and each vertex \(x_i \in V(C)\) has two corresponding vertices \(x'_i\) and \(x''_i\) in \(G'\). 
Finally, identifying \(x'_i\) and \(x''_i\) for each \(i \in [k]\) provides a graph homomorphism mapping \(G'\) back to \(G\). Notice that (each component of) \(\Pi'\) is contoured by \(G'\), provided that \(\Pi\) is contoured by \(G\).

\begin{figure}[h]
\centering
\includegraphics[height=0.2\textwidth]{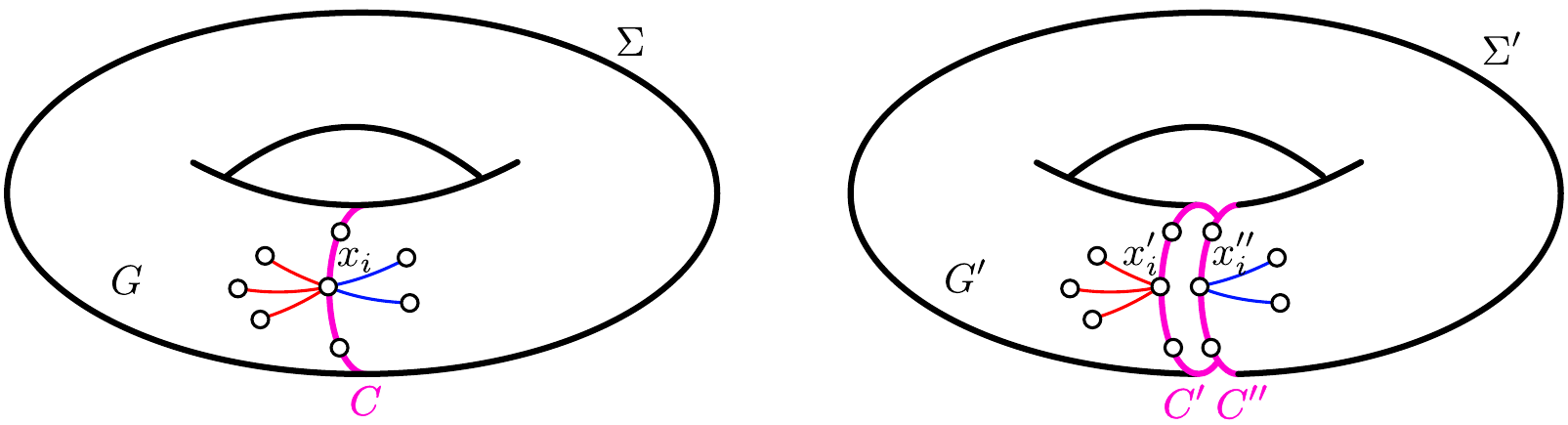}
\caption{Cutting along cycle \(C\) (two-sided case). To the left, graph \(G\) and surface \(\surf\) before cutting. To the right, graph \(G'\) and surface with boundary \(\surf'\) after cutting.}
\label{fig:cutting_along_C}
\end{figure}

The operation of cutting along a single cycle \(C\) generalizes in a straightforward way to \emph{cutting along a set \(\CC\) of (vertex-) disjoint cycles}. This can be performed for instance by iteratively cutting along the cycles of \(\CC\), in any order. 

Consider a graph $G$ embedded in a surface $\surf$, and a collection $\CC$ of disjoint cycles in $G$. Let $\surf'$ denote the disjoint union of surfaces with boundary obtained by cutting $\surf$ along $\CC$. We define a \emph{canonical projection} map $\pi:\surf' \to \surf$ as the natural inverse of cutting, where we map any point of $\surf'$ back to its preimage in $\surf$. This means in particular, that if some point $x \in \surf$ is on a cycle $C$ which we cut along, and split into two points $x',x'' \in \surf'$ by cutting along $C$, these points are mapped back to $x$ by $\pi$, that is, $\pi(x')=\pi(x'')=x$.

Given a graph $G$ embedded in a fixed surface $\surf$, and a surface with boundary $\Pi \subseteq \surf$ contoured by $G$, we denote by $G \cap \Pi$ the subgraph of $G$ embedded in $\Pi$.

\section{The planar case} \label{sec:planar}

The goal of this section is to prove \Cref{thm:planar}. We repeat the definition of rooted $K_{2,t}$-model from the introduction, with a more systematic approach and slightly different notations. 

Given a graph $H$, an {\em $H$-model $M$} in a graph $G$ consists of one connected\footnote{If $G$ is a graph and $X \subseteq V(G)$, we say that $X$ is \emph{connected} if the induced subgraph $G[X]$ is.} vertex subset $M(v) \subseteq V(G)$ for each $v\in V(H)$, such that $M(u)$ and $M(v)$ are disjoint for all distinct $u, v \in V(H)$, and one edge $M(uv) \in E(G)$ with one endpoint in $M(u)$ and the other in $M(v)$ for each edge $uv \in E(H)$.
The sets $M(v)$ for $v \in V(H)$ are called {\em branch sets}.
 
 Let $(G, R)$ denote a rooted graph. Let $H \cong K_{2,t}$ denote the graph with $V(H) = \{x_1,x_2\} \cup \{y_1,\ldots,y_t\}$ and $E(H) = \{x_iy_j \mid i \in [2],\ j \in [t]\}$. Consider a $H$-model $M$ in $G$. We say that $M$ is {\em (properly) rooted} if each of the $t$ branch sets $M(y_1)$, \ldots, $M(y_t)$ contains a vertex of $R$. In this case, we say that $(G,R)$ has a {\em rooted $K_{2,t}$-model}, or equivalently that $(G,R)$ has a {\em rooted $K_{2,t}$ minor}. 
 
Let $(G,R)$ denote a rooted graph, let $T$ and $T'$ denote two edge-disjoint trees in $G$, each with at least $3$ vertices, and let $S := V(T) \cap V(T')$. Assume that every vertex in $S$ is simultaneously a leaf in $T$, a leaf in $T'$, and a member of $R$.
Then $G$ has rooted $K_{2,|S|}$-model built from $T$ and $T'$. We say that $T$ and $T'$ are edge-disjoint trees \emph{touching} at $S$.

Before turning to the proof of \Cref{thm:planar}, we state a result of Bienstock and Dean that we invoke below to reduce to the case where no two roots are incident to the same face. 

\begin{theorem}[Bienstock and Dean~\cite{BD92}] \label{thm:BD92}
Let $(G,R)$ be a plane rooted graph, let $\nu$ denote the maximum number of roots of $(G,R)$ which are pairwise not incident to a common face, and let $\tau$ denote the minimum size of a face cover. Then $\tau \le 27 \nu$.
\end{theorem}

We point out that the constant $27$ in \Cref{thm:BD92} can be improved to $8.38$ using a recent result by Schlomberg~\cite{Schl24}.

\begin{proof}[Proof of \Cref{thm:planar}]
We prove the contrapositive: assuming that $(G,R)$ has no face cover of size at most $f_{\ref{thm:planar}}(t) := 27 t^4$, we want to show that $(G,R)$ has a rooted $K_{2,t}$-model.
By \Cref{thm:BD92}, we can extract from $R$ a subset $R'$ of $t^4$ roots such that no two distinct roots of $R'$ are incident to the same face of $G$. For simplicity, we assume that $R = R'$ below and aim to construct a rooted $K_{2,t}$-model in the case where we have $|R| = t^4$ roots, no two of them incident to the same face. 

Consider any Schnyder embedding of $G$ in $\Delta = \{x \in \R^3_{\ge 0} : x_1 + x_2 + x_3 = 1\}$, see~\Cref{thm:Schnyder_embedding}. Let $T_1$, $T_2$ and $T_3$ denote the three in-arborescences for our embedding. For the sake of simplicity, we assume that $G$ coincides with its embedding. Hence, $a_1 = (1,0,0)$, $a_2 = (0,1,0)$ and $a_3 = (0,0,1)$ are the three special vertices of $G$ relative to which our Schnyder embedding is defined. 

For $i \in [3]$, consider the poset $P_i = (R,\leq_i)$ such that $u \leq_i v$ if and only if $u = v$, or $u_{i-1} < v_{i-1}$ and $u_{i+1} < v_{i+1}$ (recall that indices are computed cyclically). Geometrically, we have $u \leq_i v$ whenever $u$ is in the parallelogram with corners $a_i$ and $v$, but not on the two sides incident to $v$.

We remark that every $3$-connected rooted graph $(G,R)$ with at least one root has a rooted $K_{2,1}$-model: take any root and two of its neighbors. Moreover, every $3$-connected rooted graph $(G,R)$ with at least two roots has a rooted $K_{2,2}$-model: letting $r_1$ and $r_2$ denote two distinct roots, consider three internally disjoint $r_1$--$r_2$ paths; at least two of them have a nonempty interior, yielding the desired rooted $K_{2,2}$-model. Hence, below we may assume that $t \geq 3$, whenever necessary.

\begin{claim} \label{claim:first}
If there is an index $i \in [3]$ and a set of roots $S \subseteq R$ that is a chain in $P_i$, then $(G,R)$ has a rooted $K_{2,|S|}$-model.
\end{claim}

\begin{proofofclaim}
Let $s := |S|$. By the above remark, we may assume that $s \ge 3$. It follows that $a_{i-1}, a_{i+1} \notin S$. Recall that, for $j \in [3]$, the parallelogram with corners $a_j$ and $u$ is defined as $\Pi_j(u) := \{x \in \Delta : x_{j-1} \le u_{j-1},\ x_{j+1} \le u_{j+1}\}$, see \Cref{sec:background_Schnyder} above. 
For $j \in [3]$, let $\Pi_j(S) := \bigcup_{u \in S} \Pi_j(u)$ denote the union of all parallelograms whose corners are $a_j$ and some $u \in S$.

Notice that $\Pi_{i-1}(S)$ and $\Pi_{i+1}(S)$ only meet in $S$. Indeed, for all $u \in V(G)$, $\Pi_{i-1}(u) \cap \Pi_{i+1}(u) = \{u\}$. And if $u$ and $v$ are distinct, comparable elements of $P_i$, then $\Pi_{i-1}(u)$ and $\Pi_{i+1}(v)$ are disjoint, as well as $\Pi_{i+1}(u)$ and $\Pi_{i-1}(v)$, see \Cref{fig:BM1}.

\begin{figure}[h]
\centering
\includegraphics[width=0.4\textwidth]{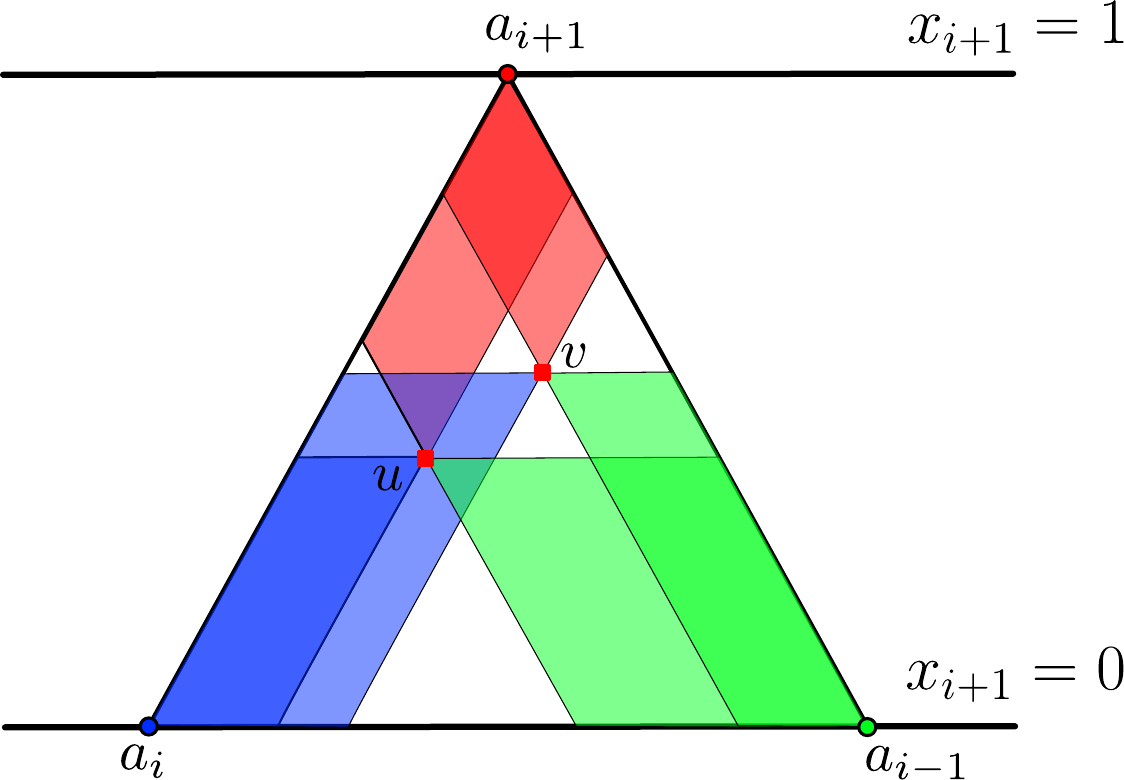}
\caption{Roots $u$ and $v$ such that $u <_i v$, together with the three parallelograms with corner $u$ and the three parallelograms with corner $v$.}
\label{fig:BM1}
\end{figure}

For $j \in \{i-1,i+1\}$, let $\tilde{T}_{j} \subseteq \Pi_{j}(S)$ denote the subtree of $T_{j}$ induced by the vertices in $S$ and their ancestors in $T_{j}$. Since $\Pi_{i-1}(S) \cap \Pi_{i+1}(S) = S$ and since $a_{i-1}, a_{i+1} \notin S$, we conclude that $T := \tilde{T}_{i-1}$ and $T' := \tilde{T}_{i+1}$ are two edge-disjoint trees of $G$ touching at $S$. Hence, $(G,R)$ has a rooted $K_{2,s}$-model. \QEDclaim
\end{proofofclaim}

\begin{claim} \label{claim:second}
If there is an index $i \in [3]$ and a set of roots $S \subseteq R$ such that $u_i$ is constant for all $u \in S$, then $(G,R)$ has a rooted $K_{2,|S|}$-model.
\end{claim}

\begin{proofofclaim}
Let $c$ denote the constant such that $u_i = c$ for all $u \in S$, and let $s := |S|$. As in the previous claim, we may assume without loss of generality that $s \ge 3$. In particular, $c < 1$. Since no two roots are incident to the same face, we may assume furthermore that $c > 0$.

We let $T := \tilde{T}_{i}$ denote the subtree of $T_i$ induced by the roots in $S$ and all their ancestors in $T_i$. Notice that $T \subseteq \Pi_i(S) \subseteq \{x \in \Delta : x_i > c\} \cup S$. 

In order to get our second tree $T'$, we define for each root $u \in S$ a special path $P = P(u)$ connecting $u$ to, say, $a_{i-1}$. The path $P$ is constructed iteratively as follows, see \Cref{fig:BM2} for an illustration. Initially, we start at $v = u$. At each iteration, we check whether the current vertex $v$ has a neighbor $w$ with $w_i < c$. If this is the case, we add the edge $vw$ to the path $P$ and move to vertex $w$. Otherwise, we let $w$ denote the parent of $v$ in $T_{i-1}$, add the edge $vw$ to $P$ and move to $w$. In this latter case we have $v_i = w_i = c$, so the edge is horizontal in \Cref{fig:BM2}. Once the current vertex $v$ has $v_i < c$, path $P$ simply proceeds by following $T_{i-1}$ all the way to $a_{i-1}$. Since there are no two roots incident to the same face, path $P = P(u)$ enters $\{x \in \Delta : x_i < c\}$ before it visits any other root $u' \in S$. The tree $T'$ is simply the union of all paths $P(u)$ for $u \in S$.

By construction, $T \subseteq \{x \in \Delta : x_i > c\} \cup S$ and $T' \subseteq \{x \in \Delta : x_i \le c\}$. Hence, $T$ and $T'$ are two edge-disjoint trees touching at $S$, and $(G,R)$ has a once again a rooted $K_{2,s}$-model.
\QEDclaim
\end{proofofclaim}

\begin{figure}[h]
\centering
\includegraphics[width=0.4\textwidth]{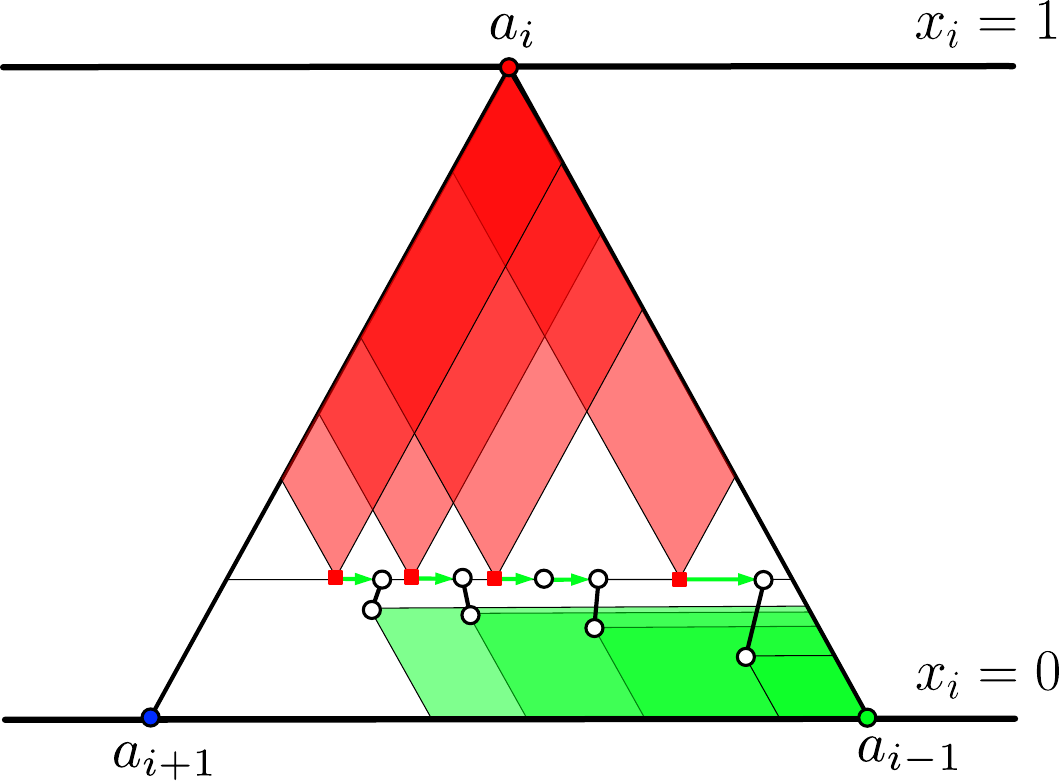}
\caption{Set of roots $S$ such that $u_i = c$ where $c$ is a constant, shown with $\Pi_i(S)$ (red parallelograms) and $\Pi_{i-1}(V)$ where $V := \{v(u) : u \in S\}$ and $v(u)$ is the first vertex of $P(u)$ that has $v_i < c$ (green parallelograms).}
\label{fig:BM2}
\end{figure}

To conclude the proof, we use three times in a row Mirsky's theorem: each poset of height $h$ can be partitioned into $h$ antichains.
Recall that the ground set of each poset $P_i$ is the root set $R$. If there is an index $i \in [3]$ such that $P_i$ contains a chain of size $|R|^{1/4} = t$, we apply \Cref{claim:first}. Now, assume that no $P_i$ contains a chain of size $t$. Then, one can find an antichain $S_1$ of size $|R|^{3/4}=t^{3}$ in $P_1$. Next, one can find a subset $S_2$ of $S_1$ with size $|R|^{1/2} = t^2$ which is an antichain in $P_2$. Finally, one can find a subset $S := S_3$ of $S_2$ with size $|R|^{1/4} = t$ which is an antichain in $P_3$. If $|S| \ge 4$, then it is easy to see that the vertices $u \in S$ are such that $u_i$ is constant for some $i \in [3]$. Hence, \Cref{claim:second} applies. The remaining case $|S| \le 3$ can be treated by adapting the arguments of \Cref{claim:second} in a ad-hoc fashion.
\end{proof}

For $t \in \mathbb{Z}_{\ge 6}$ even, let the {\em windmill of parameter $t$} be the $3$-connected, planar, rooted graph $W_t$ defined as follows. Let $p := t/2$ and $q := \lfloor t/5 \rfloor$. We start with a cycle $u_1v_1u_2v_2 \cdots u_{p}v_{p}u_1$ of length $2p = t$, and add to it a vertex $z$ adjacent to every vertex of the cycle. Then, for each $i \in [p]$, we build a vane of the windmill by adding three paths $a_{i,1}a_{i,2} \cdots a_{i,2q}$, $b_{i,1}b_{i,2} \cdots b_{i,2q}$ and $c_{i,1}c_{i,2} \cdots c_{i,2q}$, where $a_{i,1} = u_i$, $b_{i_1} = v_i$ and $c_{i,1} = u_{i+1}$ (the other vertices being distinct from those added to the graph so far). Next, for each $i \in [p]$ and each $j \in [2q]$ with $j \geq 2$, we add the edges $a_{i,j}b_{i,j}$ and $b_{i,j}c_{i,j}$. Also, for each $i \in [p]$, we add the edge $a_{i,2q}c_{i,2q}$. Finally, for each $i \in [p]$ and each odd $j \in [2q]$, the vertex $b_{i,j}$ is marked as a root.

\begin{proposition} \label{prop:windmill}
For all $t \ge 6$ even, the windmill $W_t$ has no rooted $K_{2,t}$ minor, but every face cover of $W_t$ has size $\Omega(t^2)$.
\end{proposition}

\begin{proof}
Observe that $W_t$ has $pq = \frac{t}{2} \lfloor \frac{t}{5} \rfloor = \Omega(t^2)$ vertices marked as roots, no two of which are incident to the same face. Thus, every face cover of $W_t$ has size at least $\Omega(t^2)$.

Suppose, for the sake of contradiction, that $W_t$ has a rooted $K_{2,t}$-model $M$. That is, we have $t + 2$ disjoint connected sets $M(x_1)$, $M(x_2)$, $M(y_1)$, \ldots, $M(y_t)$ such that there is an edge $M(x_iy_j)$ between $M(x_i)$ and $M(y_j)$ for all $i \in [2]$ and all $j \in [t]$, and such that each one of the \emph{central} branch sets $M(y_1)$, \ldots, $M(y_t)$ contains a root.

Let $C \subseteq W_t$ denote the boundary of the outer face. First, suppose that at least three central branch sets contain a vertex of $C$. After permuting the central branch sets, we may assume that $M(y_j)$ contains a vertex $r_j \in V(C)$, for each $j \in [3]$. By redefining the root set of $W_t$ to $\{r_1,r_2,r_3\}$, we obtain a rooted planar graph with a rooted $K_{2,3}$ model where all the roots are incident to the outer face. This yields a planar drawing of $K_{3,3}$, a contradiction. 

Hence, we may assume that at most two central branch sets contain a vertex of $C$. By deleting at most two branch sets from $M$ and reindexing if necessary, we modify $M$ to a rooted $K_{2,t-2}$ model where each central branch set is disjoint from $C$. 

By enlarging $M(x_1)$ and $M(x_2)$ we can make sure that $M$ covers the whole vertex set of $W_t$. This yields a partition of $V(C)$ into two \emph{connected} sets $X_1 := M(x_1) \cap V(C)$ and $X_2 := M(x_2) \cap V(C)$ (possibly one is empty and the other the whole vertex set of $C$). The fact that $X_1$ and $X_2$ are connected follows easily from planarity (for a detailed argument, see \cite[Lemma 74]{AFJKSWY25}). Notice that at most two vanes contain vertices from both $X_{1}$ and $X_{2}$. It follows that the number of central branch sets of $M$ is bounded by $(p - 2) + 2q \le \frac{t}{2} + 2 \frac{t}{5} -2 < t - 2$, a contradiction.
\end{proof}

\section{The bounded genus case} \label{sec:bounded_genus}

In this section, we prove \Cref{thm:bounded_genus}. Our general approach is to reduce to the planar case and apply \Cref{thm:planar}. We treat the projective planar case separately in \Cref{sec:projective_plane} relying in particular on a result of Gitler, Hliněný, Leaños and Salazar~\cite{GHLS08} concerning unavoidable minors in projective plane graphs with large face-width. Section \ref{sec:general} treats surfaces of Euler genus $g \geq 2$, with the help of a planarization result due to Yu~\cite{Yu97}.

\subsection{The projective planar case} \label{sec:projective_plane}

Let $r \ge 3$ be an integer. The \emph{diamond grid} $D_r$ is the planar graph whose vertices are all ordered pairs of integers $(i,j) \in \mathbb{Z}^2$ such that $|i| + |j| \le r$, $i + r \equiv j \equiv 1 \pmod{2}$, and whose edges are all unordered pairs of vertices $(i,j)(i',j')$ such that $|i-i'| + |j-j'| = 2$.

The \emph{projective diamond grid} $P_r$ is obtained by identifying as a single vertex every pair of vertices $(i,j)$ and $(-i,-j)$ such that $|i| + |j| = r$. (Notice that $P_r$ may contain a few parallel edges.) As is easily seen, this graph can be embedded in the projective plane obtained from the closed disk $\Delta = \{(x,y) \in \mathbb{R}^2 \mid |x| + |y| \le r\}$ by identifying all pairs of diametrically opposite points $(x,y)$ and $(-x,-y)$ on its boundary.  

\begin{theorem}[Gitler \emph{et al.}~\cite{GHLS08}] \label{thm:GHLS}
Let $G$ be a graph embedded in the projective plane with face-width at least $r$. Then $G$ has the projective diamond grid $P_r$ as a surface minor.
\end{theorem}

The aim of this section is to establish the special case $g = 1$ of \Cref{thm:bounded_genus}. Our strategy is to invoke \Cref{thm:GHLS} to argue that $G$ contains a sufficiently fine projective diamond grid, and then use this grid to define a surface cover of $\surf$ (see \Cref{def:surface_cover}) of size $k = 3$. Roughly speaking, we find disks $\Pi_1$, \ldots, $\Pi_k$ contoured by $G$, making sure that each $\Pi_i$ is ``protected'' by a slightly larger disk $\Pi^+_i$ also contoured by $G$. (More precisely, we require that $(\Pi_i,\Pi^+_i)$ is a nested pair of disks, see \Cref{def:nested_pair}).

For each $i \in [k]$, we construct a planar minor $H_i$ of $G$ that agrees with $G$ in the neighborhood of $\Pi_i$ (see \Cref{def:agree}), with root set $R_i := R \cap \Pi_i$. We show that $H_i$ is $3$-connected, and that every face cover of $(H_i,R_i)$ can be lifted to a face cover of $(G,R_i)$ of the same size. 

Hence, we can invoke \Cref{thm:planar} to cover the roots of $(G,R_i)$ with a small number of faces, for each $i \in [k]$. Taking the union of these $k$ face covers yields a face cover of $(G,R)$ of size at most $k \cdot f_{\ref{thm:planar}}(t) = O(t^4)$.

We point out that the above strategy is essentially the one we utilize for any Euler genus $g \geq 1$. The main difference lies in the way we construct the surface cover. This is why a large part of the definitions and results in this section are stated in full generality.

\begin{definition} \label{def:nested_pair}
    Let $G$ be a graph embedded in a fixed surface $\surf$.
    Let $\Pi,\Pi^+\subseteq\surf$ be surfaces with boundary contoured by $G$.
    We say that $\Pi$ and $\Pi^+$ are \emph{nested} if $\Pi$ is contained in $\intt(\Pi^+)$, and $\Pi^+ \setdiff \intt(\Pi)$ is a disjoint union of spheres with boundary, in which each component has exactly one cuff in common with $\Pi$.
\end{definition}

We point out that if $\Pi$ and $\Pi^+$ are as above, then $\Pi^+$ can be obtained from $\Pi$ by gluing a sphere with boundary (and at least one cuff) along \emph{each} cuff of $\Pi$. By 
\Cref{prop:gluing_surfaces_with_bd}, $\Pi$ and $\Pi^+$ have the same Euler genus and the same orientability, implying that $\widehat{\Pi}$ and $\widehat{\Pi^+}$ are homeomorphic. However, the number of cuffs of $\Pi^+$ and $\Pi$ can differ.

\begin{definition} \label{def:surface_cover}
    Let $G$ be a graph embedded in a fixed surface $\surf$.
    A \emph{surface cover} of $\surf$ with respect to $G$ is a collection $\left\{(\Pi_i,\Pi_i^+)\right\}_{i\in[k]}$ of pairs of nested surfaces with boundary contained in $\surf$, contoured by $G$ and such that $\bigcup_{i\in[k]} \Pi_i=\surf$.
\end{definition}

\begin{lemma}\label{proj-planar-cover}
    Let $G$ be a graph embedded in the projective plane with face-width at least $16$.
    Then there is a surface cover $\left\{(\Pi_i,\Pi_i^+)\right\}_{i \in [3]}$ of the projective plane with respect to $G$.
\end{lemma}

\begin{proof}
    By \Cref{thm:GHLS}, $G$ has $G' := P_{16}$ as a surface minor.
    Let $X := \{(i,j) \in V(G') \mid j = \pm 5\} \cup \{(i,j) \in V(G') \mid i = \pm 5,\ -5 \le j \le 5\}$.
    Notice that $S' := G'[X]$ is a cellularly embedded $K_4$-subdivision with branch vertices $v_1 := (5,5)$, $v_2 := (-5,5)$, $v_3 := (-5,-5)$ and $v_4 := (5,-5)$, see \Cref{fig:PDG16s}.
    Since $G'$ is a surface minor of $G$, and $S'$ is a surface minor of $G'$, we conclude that $S'$ is a surface minor of $G$. Since $S'$ has maximum degree $3$, we get that $G$ contains a $K_4$-subdivision $S$ that is mapped to $S'$ via the contraction and deletion operations producing $G'$ from $G$.
    The $K_4$-subdivision $S$ is cellularly embedded in the projective plane and has three faces $f_1$, $f_2$ and $f_3$, respectively corresponding to the faces of $S'$ depicted in cyan, yellow and magenta in the figure. 
    
    For each $i \in [3]$, we define $\Pi_i$ as the closure of the face $f_i$. Notice that each $\Pi_i$ is a closed disk bounded by a contractible cycle. Further, $\Pi_1 \cup \Pi_2 \cup \Pi_3$ is the whole projective plane.
    
    Now consider the bold cyan, yellow and magenta cycles depicted in \Cref{fig:PDG16s}. Those are contractible cycles $C'_1$, $C'_2$ and $C'_3$ contained in $G'$. For each $i \in [3]$, we can lift $C'_i$ to a contractible cycle $C_i$, which is mapped to $C'_i$ by the operations producing $G'$ from $G$. We define $\Pi^+_i$ as the closed disk bounded by $C_i$. Observe that $(\Pi_i,\Pi_i^+)$ is a nested pair of closed disks contoured by $G$.
    
    Hence, we obtain a surface cover $\left\{(\Pi_i,\Pi_i^+)\right\}_{i \in [3]}$.
\end{proof}
        
\begin{figure}[h!]
    \centering
    \begin{tabular}{c@{\hspace{1em}}c}
    \includegraphics[width=.35\textwidth]{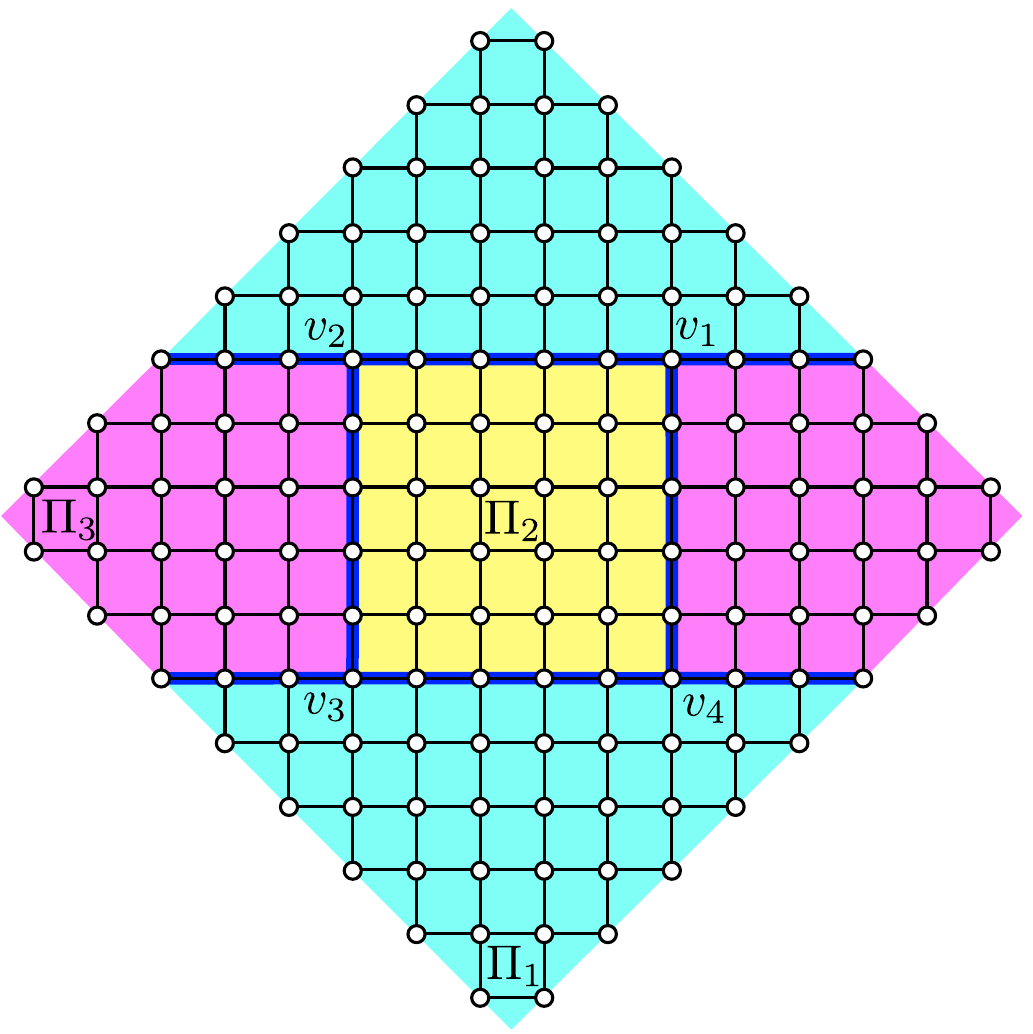}
    & \includegraphics[width=.35\textwidth]{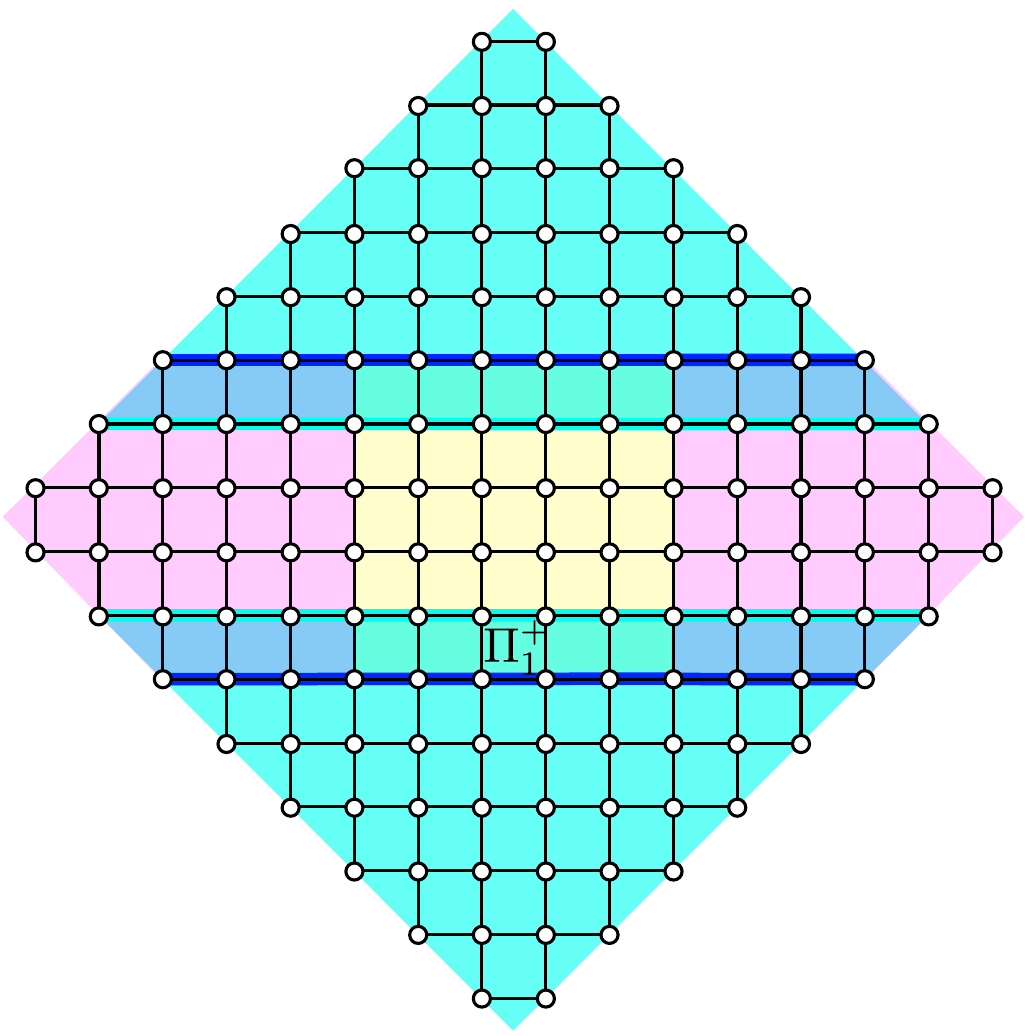}\\[1ex]
    \includegraphics[width=.35\textwidth]{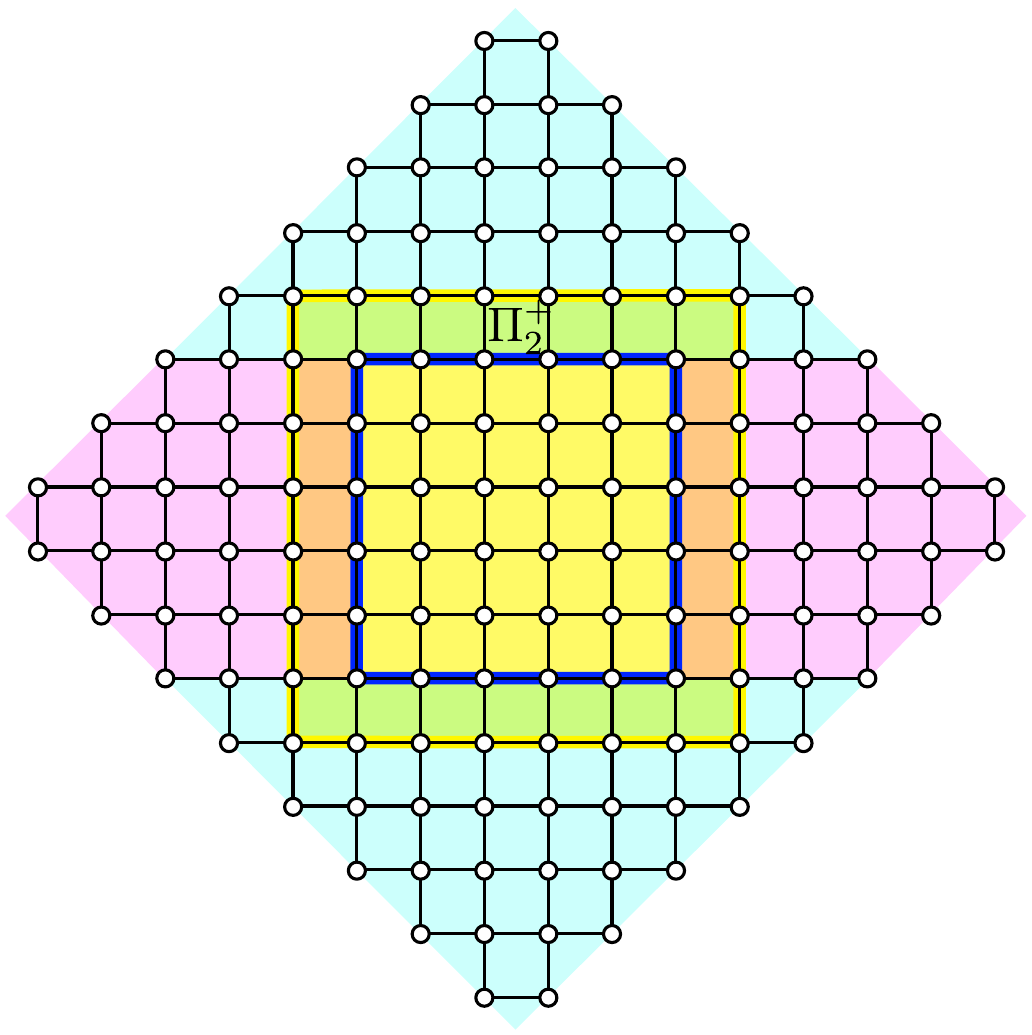}
    & \includegraphics[width=.35\textwidth]{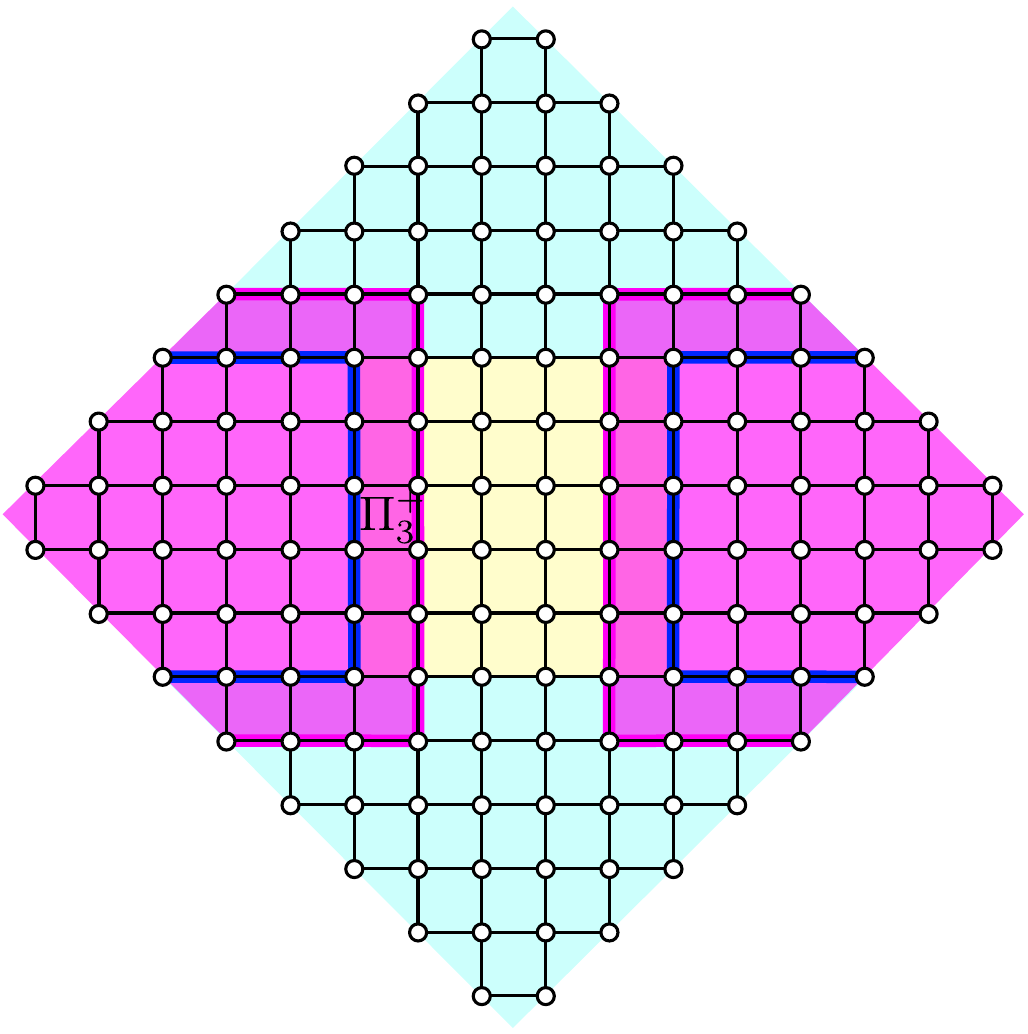}
    \end{tabular}
    \caption{The figure shows four copies of the projective diamond grids of parameter $r = 16$. In each copy, pairs of diametrically opposite points on the boundary are identified. The bold blue subgraph is a cellularly embedded $K_4$-subdivision with branch vertices $v_1$, $v_2$, $v_3$ and $v_4$. The top left copy shows $\Pi_1$, $\Pi_2$ and $\Pi_3$, which correspond to the faces of the $K_4$-subdivision. The top right copy shows $\Pi_1^+$, the bottom left $\Pi_2^+$ and the bottom right $\Pi_3^+$.    }
    \label{fig:PDG16s}
\end{figure}

Before implementing the rest of our strategy, we need to formally define what it means, for each $i \in [3]$, that the minor $H_i$ of $G$ agrees with $G$ in the neighborhood of $\Pi_i$. The definition and properties of $H_i$ are described in \Cref{3c-minor} right after.

\begin{definition} \label{def:agree}
As before, let $G$ be a graph embedded in a surface $\surf$, and let $(\Pi,\Pi^+)$ be a pair of nested surfaces with boundary contained in $\surf$ and contoured by $G$.
Let $H$ be a surface minor of $G \cap \Pi^+$, taken in $\widehat{\Pi^+}$, the surface obtained from $\Pi^+$ by capping all of its cuffs. We say that $H$ \emph{agrees with $G$ in the neighborhood of $\Pi$} if $H \cap \Pi = G \cap \Pi$ and for each face $f$ of $H$, there exists a corresponding face $\varphi(f)$ of $G$ incident to exactly the same vertices of $G \cap \Pi$.
\end{definition}

We point out that we can take $\varphi(f) = f$ in case $f$ is contained in $\Pi$, since we assume that the embeddings of $H$ and $G$ coincide within $\Pi$. The nontrivial case arises when $f$ is not contained in $\Pi$.

\begin{lemma}\label{3c-minor}
    Let $G$ be a $3$-connected graph embedded in a fixed surface $\surf$ with $\fw(G)\ge3$ and let $(\Pi,\Pi^+)$ be a pair of nested surfaces with boundary contained in $\surf$ and contoured by $G$.
    Then there exists a $3$-connected surface minor $H$ of $G\cap\Pi^+$ that agrees with $G$ in the neighborhood of $\Pi$.
\end{lemma}
\begin{proof}
    We define $H$ as the surface minor of $G\cap\Pi^+$ obtained by contracting each component of $(G\cap\Pi^+) - V(G\cap\Pi) $ into a single vertex.
    We obtain an embedding of $H$ in $\widehat{\Pi^+}$ by modifying the canonical embedding of $G\cap\Pi^+$ in this surface. Specifically, we continuously contract the components of $(G\cap\Pi^+) - V(G\cap\Pi)$ and their enclosed faces into single points.
    This contraction step can be modeled by a continuous map $\sigma: \widehat{\Pi^+}\to\widehat{\Pi^+}$, in such a way that the embedding of $H$ is the image under $\sigma$ of the embedding of $G\cap\Pi^+$.
     We choose $\sigma$ to be injective outside of the components of $(G\cap\Pi^+) - V(G\cap\Pi)$, and the identity within $\Pi$. This can be done without loss of generality.
 
    \begin{figure}[ht]
    \centering
    \includegraphics[width=.9\textwidth]{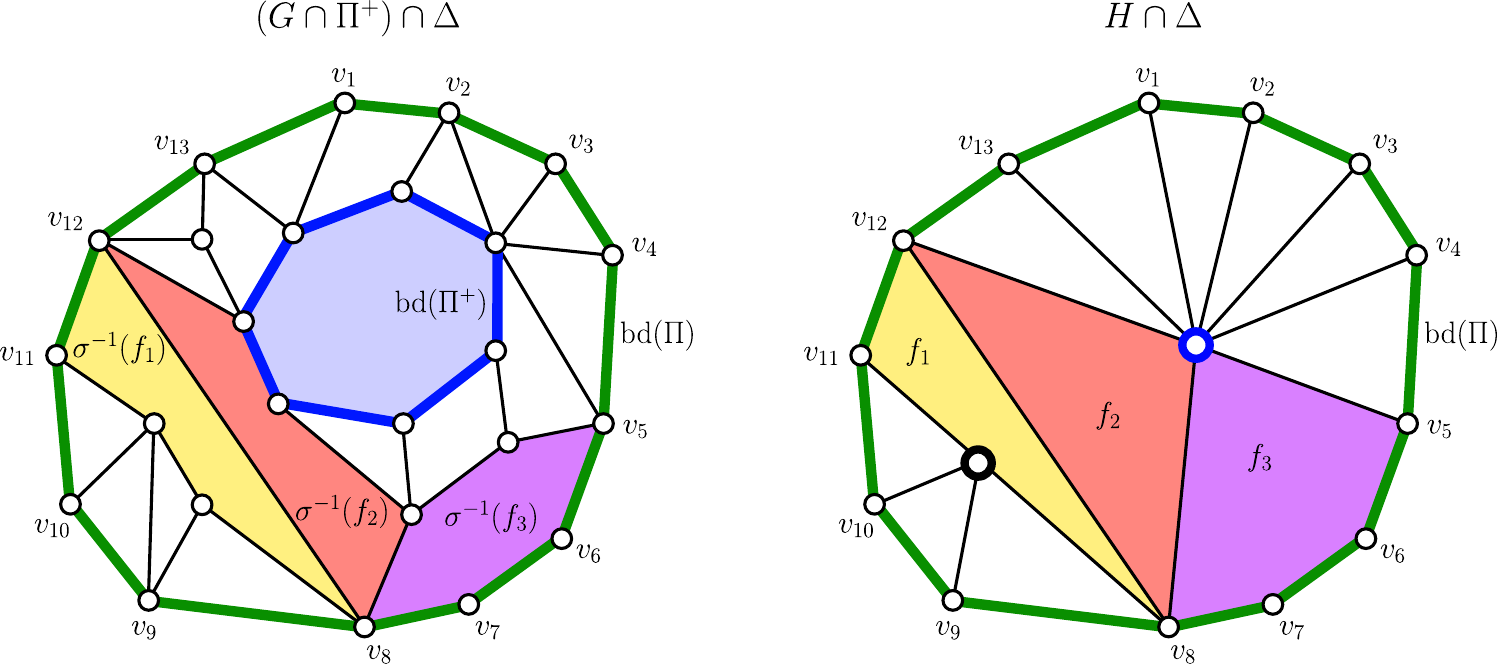}
    \caption{A graph $G\cap\Pi^+$ restricted to some disk $\Delta = \Delta_i$ which is a component of $\widehat{\Pi^+} \setdiff \intt(\Pi)$ (left), and the corresponding minor $H$ restricted to the same disk (right). The two larger vertices of $H$ correspond to component vertices. The green and the blue cycle are cuffs (that is, boundary components) of $\Pi$ and $\Pi^+$ respectively. The figure shows, for each $i \in [3]$, a face $f_i$ of $H$ contained in $\Delta$ along with and the corresponding face $\varphi(f_i) = \sigma^{-1}(f_i)$ of $G$.}
    \label{fig:nests_Klein_bottle}
    \end{figure}

    We first argue that $H$ agrees with $G$ in the neighborhood of $\Pi$. Since $\sigma$ is the identity within $\Pi$, we see immediately that $H \cap \Pi = G \cap \Pi$. Now consider some face $f$ of $H$. We let $\varphi(f) := \sigma^{-1}(f)$ denote the corresponding face of $G$. We proceed to check that this is indeed a face of $G$, which is moreover incident to the same vertices of $G \cap \Pi$ as $f$.
    
    If $f$ is contained in $\Pi$, then we have $\sigma^{-1}(f) = f$ and there is nothing to check. 
    
    Therefore, we may assume that $f$ is contained in $\widehat{\Pi^+} \setdiff \Pi$.    
    We observe that $\sigma^{-1}(f)$ is a face of $G$ since we obtained $H$ from $G\cap\Pi^+$ only by contracting edges. No edge or vertex was deleted.   
    Because $\sigma$ is the identity within $\Pi$, hence in particular on the boundary of $\Pi$, we see that $\sigma^{-1}(f)$ is incident to exactly the same vertices of $G \cap \Pi$ as $f$.     
        
    It remains to show that $H$ is $3$-connected.
    We begin by inspecting the \emph{component vertices} of $H$, that is, the vertices that corresponds to a contracted component of $(G\cap\Pi^+) - V(G\cap\Pi)$. We show that every component vertex $v\in V(H)$ has degree at least $3$. 
    Assume otherwise.
    
    Notice that, since $\Pi$ and $\Pi^+$ are nested, there is a collection $\left\{\Delta_i\right\}_{i\in[c]}$ of pairwise disjoint closed disks contained in $\widehat{\Pi^+}$ such that $\widehat{\Pi^+} \setdiff \intt(\Pi) = \bigcup_{i \in [c]} \Delta_i$, where $c\in\Z_{\ge1}$ is the number of cuffs of $\Pi$.
    
    Let $i\in[c]$ be such that $v$ is embedded in the interior of $\Delta_i$.
    We can find a simple closed $H$-normal curve $o$ in $\Delta_i$ intersecting $H$ only in $N(v)$ and such that $v$ is contained in the open disk bounded by $o$. In particular, $o$ avoids every component vertex of $H$, 
    Hence, $\sigma^{-1}(o)$ is a simple closed $G$-normal curve in $\surf$ contained in $\intt(\Pi^+)$, intersecting $G$ in at most two vertices, which are both on the boundary cycle of $\Delta_i$.
    If $\sigma^{-1}(o)$ is non-contractible, then it is a noose, contradicting $\fw(G)\ge3$.
    Otherwise, $\sigma^{-1}(o)$ bounds an open disk containing $v$, contradicting the $3$-connectivity of $G$.

    Next, we show that $H$ is $3$-connected. We heavily rely on \Cref{faces-cycles}, which we use repeatedly below (without explicit mention). Recall that an embedding is called polyhedral if each face is an open disk bounded by a cycle and every two facial cycles intersect in a vertex, an edge, or not at all. By hypothesis, $G$ is polyhedrally embedded in $\surf$. We show that this implies that $H$ is polyhedrally embedded in $\widehat{\Pi^+}$. This in particular establishes the $3$-connectivity of $H$.
    
    Every face of $G$ is an open disk bounded by a cycle. Since $H \cap \Pi = G \cap \Pi$, this implies that every face of $H$ contained in $\Pi$ is an open disk, and is bounded by a cycle. Now consider a face $f$ of $H$ that is not contained in $\Pi$, thus $f$ is contained in $\Delta_i$ for some $i \in [c]$.     Consider the corresponding face $\sigma^{-1}(f)$ in $G \cap \Pi^+$. We know that $\sigma^{-1}(f)$ is an open disk bounded by a cycle, which implies that $f$ also. Indeed, the boundary of $f$ can be obtained from that of its preimage $\sigma^{-1}(f)$ by contracting subpaths, which are all pairwise at distance at least $2$.

    Finally, consider two faces of $H$, say $f_1$ and $f_2$. Since $H$ agrees with $G$ in the neighborhood of $\Pi$, we may assume that $f_1$ and $f_2$ are incident to some common component vertex, say $v$. Indeed, otherwise the boundary cycles of $f_1$ and $f_2$ intersect exactly as the boundary cycles of $\sigma^{-1}(f_1)$ and $\sigma^{-1}(f_2)$, that is, in a vertex, an edge or not at all.
    
    Hence, $f_1$ and $f_2$ are contained in the same $\Delta_i$, for some $i \in [c]$. Let $C \subseteq H$ denote the boundary of $\Delta_i$. Notice that $N_H(v) \subseteq V(C)$. Let $W$ denote the subgraph of $H$ induced by $V(C) \cup \{v\}$. This is a subdivided wheel. Since $f_1$ and $f_2$ are both incident to $v$ and distinct, they are contained in distinct faces of $W$. Since $v$ has at least three neighbors, we conclude that the boundaries of $f_1$ and $f_2$ intersect as they should.      
    \end{proof}

\begin{lemma}\label{lifting-covers}
    Let $(G,R)$ be a rooted graph embedded in a surface $\surf$.
    Let $\left\{(\Pi_i,\Pi_i^+)\right\}_{i\in[k]}$ be surface cover of $\surf$ with respect to $G$, and let $(H_i)_{i\in[k]}$ be a collection of corresponding minors, such that $H_i$ agrees with $G$ in the neighborhood of $\Pi_i$ for each $i \in [k]$. 
    Let $R_i := R \cap \Pi_i$ denote the roots contained in $\Pi_i$ for each $i \in [k]$.
    If each $(H_i,R_i)$ admits a face cover of size $\phi_i\in\N$, then $(G,R)$ admits a face cover of size at most $\sum_{i\in[k]}\phi_i$.
\end{lemma}

\begin{proof}
    For each root $r\in R$, there is some $i = i(r) \in[k]$, such that $r\in\Pi_i$.
    Therefore, $r\in R_i$ and it is covered by some face $f$ in the face cover of $(H_i,R_i)$.
    Since $r\in\Pi_i$ and $H_i$ agrees with $G$ in the neighborhood of $\Pi_i$, we can find a face $\varphi_i(f)$ of $G$ covering the same set of roots within $\Pi_i$. In particular, $\varphi_i(f)$ covers $r$.
    We add $\varphi_i(f)$ to the face cover and repeat this process for every root in $R$.
    We obtain a face cover of $(G,R)$, which is the union for $i \in [k]$ of the image under $\varphi_i$ of the face cover of $(H_i,R_i)$. The result follows.
\end{proof}

\begin{proposition}\label{proj-planar-case}
    Let $(G,R)$ be a $3$-connected rooted graph without a rooted $K_{2,t}$ minor embedded in the projective plane with $\fw(G) \ge 16$.
    Then $(G,R)$ admits a face cover of size at most $3\cdot f_{\ref{thm:planar}}(t)\in O(t^4)$.
\end{proposition}
\begin{proof}
    By \Cref{proj-planar-cover}, we can find a surface cover $\left\{(\Pi_i,\Pi_i^+)\right\}_{i \in [3]}$ of the projective plane with respect to $G$
    and \Cref{3c-minor} allows us to find planar minors $\left\{H_i\right\}_{i\in[3]}$, where each $H_i$ agrees with $G$ in the neighborhood of $\Pi_i$.
    Therefore, by \Cref{thm:planar} we obtain a face cover of size at most $f_{\ref{thm:planar}}(t)$ for each of the three rooted minors.
    By \Cref{lifting-covers}, we obtain a face cover of size at most $3\cdot f_{\ref{thm:planar}}(t)$ for $(G,R)$.
\end{proof}

\subsection{The general case} \label{sec:general}

Let \(G\) denote a graph embedded in a surface \(\surf\). A \emph{planarizing set of cycles} is a set of disjoint cycles \(\CC\) in \(G\) such that cutting along \(\CC\) yields a graph \(G'\) embedded in a disjoint union of \emph{spheres with boundary} \(\surf'\).
Notice that $\surf'$ is contoured by $G'$.
Further, if \(\CC\) is an inclusion-wise minimal planarizing set of cycles, then \(\surf'\) has only one component, and thus is a \emph{single} sphere with boundary. 
(Indeed, otherwise there exists a two-sided cycle \(C \in \CC\) such that undoing the operation of cutting along \(C\) joins two distinct connected components of \(\surf'\), hence \(\CC \setdiff \{C\}\) is also planarizing, contradicting the minimality of \(\CC\).)
If \(\CC'\) denotes the set of cycles in \(G'\) corresponding to the cycles from \(\CC\) in \(G\),
then we call the tuple \((\CC, G', \surf', \CC')\) a \emph{planarization} of \(G\) and \(\surf\). Notice that, whenever \(\CC\) is inclusion-wise minimal \(\CC'\) contains precisely \(g\) cycles, where \(g\) is the Euler genus of \(\surf\). \Cref{fig:cutting_along_C} above provides an example with \(\CC = \{C\}\) and \(\CC' = \{C',C''\}\).

For a graph \(G'\) embedded in a sphere with boundary \(\surf'\) contoured by \(G'\), a cuff \(C\) and an integer \(d \ge 1\), a \emph{depth-\(d\) nest} around \(C\) is a sequence of disjoint cycles \(\DD = \DD(C) = (D_0, \ldots, D_d)\) in \(G' \) and closed disks \(\Delta_0 \subsetneq \Delta_1 \subsetneq \cdots \subsetneq \Delta_d\) in \(\widehat{\surf'}\), the sphere obtained from \(\surf'\) by capping all cuffs, such that \(D_0 = C\), and for each \(i \in \{0,\ldots,d\}\), cycle \(D_i\) is the boundary of \(\Delta_i\). 

A cycle \(C''\) in \(G'\) is said to be \emph{outside} the nest \(\DD = \DD(C')\) if \(C''\) is disjoint from the disk \(\Delta_d\). Two depth-\(d\) nests \(\DD' = (D_0', \ldots, D_d')\) and \(\DD'' = (D_0'', \ldots, D_d'')\) around distinct cuffs are \emph{disjoint} if \(V(D_i') \cap V(D_j'') = \emptyset\) for all \(i,j\in\{0,\ldots, d\}\). 

\begin{figure}[ht]
\centering
\includegraphics[width=0.6\textwidth]{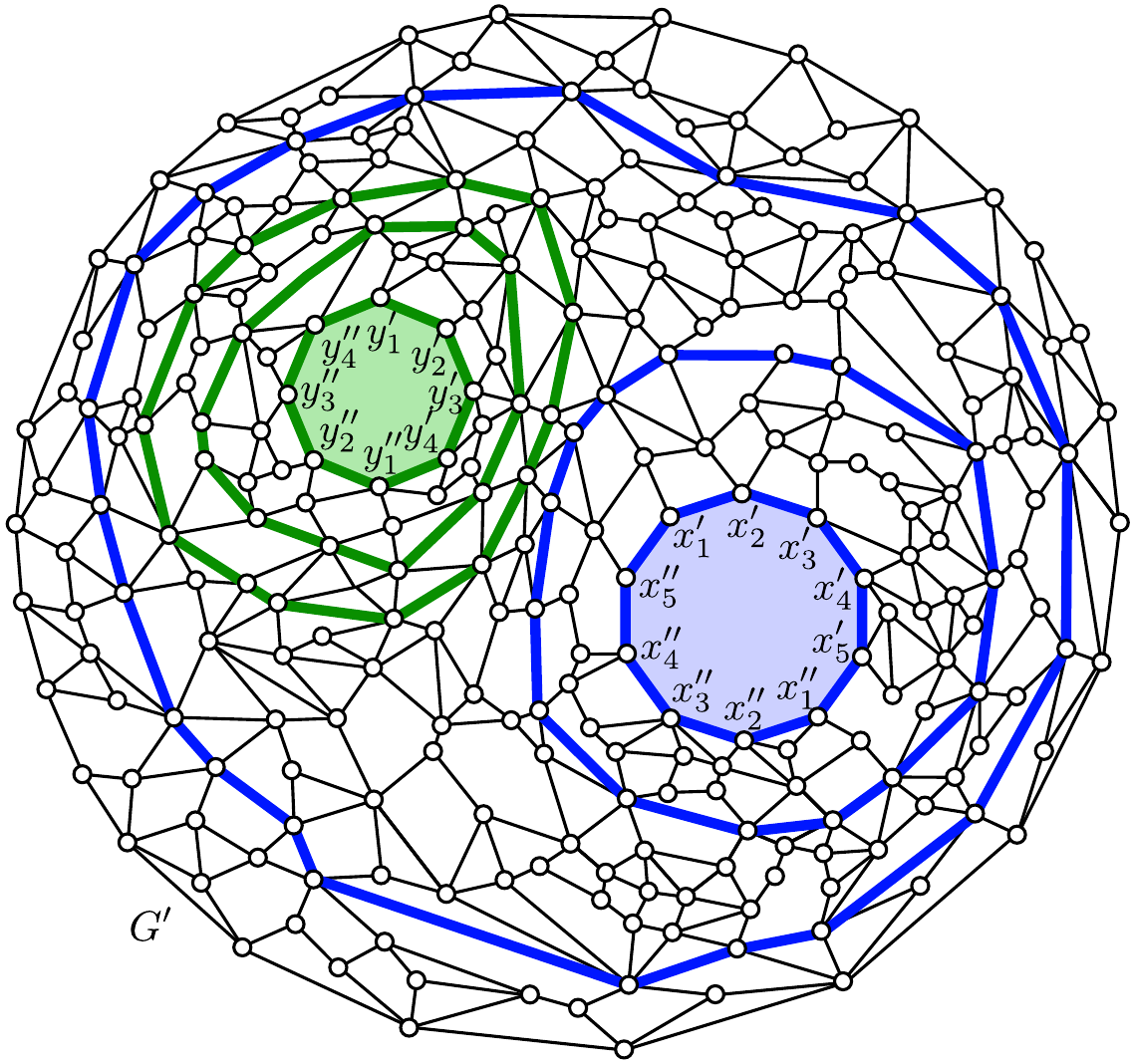}
\caption{A graph \(G'\) embedded in a sphere with boundary \(\surf'\), obtained by planarizing a graph \(G\) embedded on the Klein bottle \(\surf\), along two one-sided cycles. The two cuffs of \(\surf'\) are the boundaries of the shaded disks.}
\label{fig:nests_Klein_bottle}
\end{figure}

\begin{theorem}[Yu~{\cite[Thm.~4.3]{Yu97}}]\label{Yu}
 Let \(g, d \ge 1\) be integers. 
 Let \(\surf\) be a surface with Euler genus \(g\), and let  \(G\) be a graph embedded in \(\surf\) with face-width at least \(w_{\ref{Yu}}(g, d) := 8(2d+1)(2^{g}-1)\). 
 Then there exists a planarization \((\CC, G', \surf', \CC')\) of \(G\) and a depth-\(d\) nest \(\DD(C')\) in \(G'\) (which is embedded in \(\surf'\)) around each \(C' \in \CC'\) such that for any distinct \(C', C'' \in \CC'\), the nests \(\DD(C')\) and \(\DD(C'')\) are disjoint and \(C'\) is outside \(\DD(C'')\) or \(C''\) is outside \(\DD(C')\).
\end{theorem}

See \Cref{fig:nests_Klein_bottle} for an illustration of Yu's planarization theorem.\medskip

Below, we construct from the planarization and the nests produced by \Cref{Yu} a collection of at most $2g$ surfaces with boundary covering $\surf$. We make sure that each surface with boundary in the cover is either a sphere with boundary or a projective plane with boundary. Hence, we can apply \Cref{thm:planar} or \Cref{proj-planar-case} to each one.

Let \(g \ge 2\) be an integer.
Let \(r := 16\), let \(d := 3r+1\), and let \(w(g) := \max\{w_{\ref{Yu}}(g, d), r+1\}\).
Let \(\surf\) be a surface with Euler genus \(g\), and let \(G\) be a graph embedded in \(\surf\) with face-width at least \(w(g)\).
We apply~\Cref{Yu}. Denote by \(\CC' = \{C'_1,\ldots,C'_g\}\) the set of cuffs of \(\surf'\), regarded as cycles of \(G'\). For each \(i \in [g]\), let \(\DD(C'_i) = (D_{i,0},\ldots,D_{i,d})\) denote the nest around \(C'_i\), and by \(\Delta_{i,0} \subsetneq \cdots \subsetneq \Delta_{i,d}\) the corresponding (closed) disks. Recall that $\widehat{\surf'} = \surf' \cup \Delta_{1,0} \cup \cdots \cup \Delta_{g,0}$ is the sphere obtained from \(\surf'\) by capping all its cuffs.

\paragraph{Construction of an auxiliary tree $T$.}
Consider the subgraph $H$ of $G'$ 
obtained by taking the union of all the cycles $D_{i,j}$ where $i \in [g]$ and $j \in \{0\} \cup [d]$. 
Recall that these cycles are pairwise disjoint. 
We obtain a graph embedded on the sphere $\widehat{\surf'}$, whose dual graph $H^*$ is a tree with parallel edges. 
Let $T$ denote the tree obtained from $H^*$ by keeping one representative edge in each class of parallel edges. 
We make $T$ into a rooted tree by marking as the root the face which is the complement of $\Delta_{1,d} \cup \cdots \cup \Delta_{g,d}$ in $\widehat{\surf'}$. 
We denote the root of $T$ by $r$ and the collection of non-root leaves of $T$ by $L$. 

Note that the tree $T$ has exactly $g(d+1)+1$ vertices and $g$ non-root leaves.
By construction of $T$ through the dual, each vertex $v\in V(T)$ bijectively corresponds to a face $f(v)$ of $H$ embedded in $\widehat{\surf'}$. For each $i\in[g]$, we denote by $\ell_i \in L$, the unique leaf with $f(\ell_i)=\intt(\Delta_{i,0})$.
Note that these faces are contained in $\widehat{\surf'} - \surf'$. Actually, $\widehat{\surf'} \setdiff \surf' = \bigcup_{i \in [g]} f(\ell_i)$.
Further, each edge $e\in E(T)$ corresponds to a cycle of $G'$ in the following way.
Let $e\in E(T)$ be the representative edge for the edges which are dual to the edges of the cycle $D_{i,j}$ for some $i\in [g]$ and $j\in\{0\}\cup[d]$.
Then we label the edge $e$ with the cycle $D(e)=D_{i,j}$. 

Every subset of vertices $U \subseteq V(T)$ has a natural mapping to a subspace of $\surf'$. We define $\Pi'(U)$ to be the union of all the faces $f(u)$ for $u \in U \setdiff L$ and all the cycles $D(e)$ for an edge $e \in E(T)$ incident to some vertex in $U$.
For $v \in V(T - L)$, subspace $\Pi'(\{v\})$ is the closure of a face of $H$, hence a sphere with boundary. Our next result generalizes this to connected subsets of $V(T - L)$. Moreover, the number of cuffs of $\Pi'(\{v\})$ is simply the degree of $v$ in $T$.
 
\begin{lemma} \label{lem:Pi'}
If $U \subseteq V(T - L)$ is non-empty and connected, then $\Pi'(U)$ is a sphere with boundary contained in $\surf'$. Moreover, $\Pi'(U)$ is contoured by $H = \bigcup_{i,j} D_{i,j}$ and the cuffs of $\Pi'(U)$ are the cycles $D(e)$ for $e \in \delta_T(U \setdiff L)$.
\end{lemma}

\begin{proof}
Let $u_1$, \ldots, $u_p$ be an enumeration of the vertices of $U$ such that, for all $i \in [p]$, $\{u_1,\ldots,u_i\}$ is connected and $u_i$ has (at most) one neighbor in $\{u_1,\ldots,u_{i-1}\}$. We use \Cref{prop:gluing_surfaces_with_bd} to show that $\Pi(\{u_1,\ldots,u_i\})$ is a sphere with boundary for all $i \in [p]$, by induction on $i$. The last part of the statement follows readily from the definition of $\Pi'(U)$.
\end{proof}

We denote $\Pi(U):=\pi(\Pi'(U))$ the corresponding subspace of $\surf$, where $\pi$ is the canonical projection map $\pi$ from \Cref{sec:background_surfaces}. In \Cref{set-props} below, we show that, for our very particular choice of $U$, this subspace is either a sphere with boundary or a projective plane with boundary. 

\paragraph{Using $T$ to find a surface cover.}
We partition the vertices of $T$ into connected subsets.
Let $i\in [g]$. 
We define $A_i$ as the subset of $V(T)$ containing all vertices whose distance to $\ell_i$ is at most $d/3$. (In particular, $A_i$ contains $\ell_i$.)
For distinct disks $\Delta_{i,0}$ and $\Delta_{j,0}$, \Cref{Yu} gives us that either $\Delta_{i,0}$ is disjoint from $\Delta_{j,d}$, or $\Delta_{j,0}$ is disjoint from $\Delta_{i,d}$.
This implies that every two distinct vertices of $L \cup \{r\}$ have distance at least $d$. Hence, $A_{i}$ and $A_{j}$ have distance at least $d/3$, whenever $i, j \in [g]$ are distinct. Let $B_1,B_2,\ldots, B_k$ be the vertex sets of the connected components of $T - A_1 - \cdots - A_g$. 

We proceed to argue that $k\le g$.
We define a surjective map $\psi:[g]\to[k]$ in the following way.
For each $i \in [g]$, we walk up in $T$ towards the root starting from $\ell_i$ until we leave $A_i$. By what precedes, we enter $B_j$ for some $j \in [k]$. We let $\psi(i) := j$. Observe $\psi(i)$ is well defined since the root of $T$ is not contained in $A_i$ for any $i \in [g]$. To see that $\psi$ is surjective, pick any $j \in [k]$ and any vertex $u \in B_j$ and walk down from there until we leave $B_j$. Let $v$ be the first vertex we encounter outside of $B_j$. We know that $v$ is contained in $A_i$ for some $i \in [g]$. We continue by following the path from $v$ down to $\ell_i$, which completely lies within $A_i$. We obtain a path going down from $u$ to $\ell_i$, whose first part is contained in $B_j$ and second part in $A_i$. This gives an index $i \in [g]$ such that $\psi(i) = j$. We conclude that $k\le g$.

Let $L_1\subseteq L$ be the collection of non-root leaves of $T$ such that the corresponding cuff is a one-sided cycle in $\surf$.
We let $\mathcal{A}_1 := \{A_i \mid i\in [g],\ \ell_i\in L_1\}$.
The rest of the $A_i$'s can be partitioned into pairs where the corresponding leaves correspond to cuffs resulting from cutting along the same two-sided cycle in $\surf$. Let $M \subseteq {L \setdiff L_1 \choose 2}$ denote the corresponding matching. 
Let $\mathcal{A}_2 := \{A_{i} \cup A_{j} \mid \ell_i\ell_j \in M\}$. 
Let $\mathcal{A} := \mathcal{A}_1\cup \mathcal{A}_2$. 
Let $\mathcal{B} := \{B_j \mid j \in [k]\}$ and $\mathcal{S}:=\mathcal{A}\cup\mathcal{B}$.

For each set $S \in \mathcal{S}$, we define a set $S^+$ where $S^+$ contains $S$ as well as all adjacent vertices. Notice that $S^+ = N_T[S]$.

\paragraph{Properties of the surface cover.}
Since the members of $\mathcal{S}$ partition $V(T)$, the collection $\left\{(\Pi(S),\Pi(S^+))\right\}_{S\in\mathcal{S}}$ is a surface cover of $\surf$.
In the following, we present further properties of $S$ and $S^+$ for $S\in\mathcal{S}$ that allow us to complete the proof of \Cref{thm:bounded_genus}.

\begin{lemma}\label{set-props}
    Fix $S\in\mathcal{S}$.
    Then $(\Pi(S),\Pi(S^+))$ is a pair of nested surfaces with boundary.
    Furthermore, $\Pi(S)$ and $\Pi(S^+)$ are projective planes with boundary if $S\in\mathcal{A}_1$ and spheres with boundary otherwise.
\end{lemma}

\begin{proof}
    Let $S$ be any member of $\mathcal{S}$. We invoke \Cref{lem:Pi'} and conclude that $\Pi'(S)$ is a sphere with boundary whenever $S \in \mathcal{A}_1 \cup \mathcal{B}$, and the disjoint union of two spheres with boundary whenever $S = A_i \cup A_j \in \mathcal{A}_2$. (Recall that $A_i$ and $A_j$ are at distance at least $d/3 \ge 4$ whenever $i, j \in [g]$ are distinct.)

    Recall that $\pi$ is injective outside of the $g$ cuffs of $\Sigma'$, namely, the cycles $D_{i,0}$ for $i \in [g]$. In order to understand $\Pi(S) = \pi(\Pi'(S))$ topologically, we have to see which cuffs of $\Sigma'$ belong to $\Pi'(S)$.
    
    Observe that $\Pi'(S)$ contains at most two cuffs of $\Sigma'$. More precisely, $S$ contains exactly one cuff if $S \in \mathcal{A}_1$, two cuffs if $S \in \mathcal{A}_2$ and no cuff if $S \in \mathcal{B}$. Furthermore, we made sure that if $S$ contains one cuff, this cuff is mapped to a one-sided cycle in $\surf$ by $\pi$, and if $S$ contains two cuffs, these cuffs originate from the same two-sided cycle in $\surf$, and are both mapped to this cycle by $\pi$.
        
    We derive the following consequences. First, if $S \in \mathcal{B}$, then $\Pi(S)$ is a sphere with boundary. Second, if $S = A_i \cup A_j \in \mathcal{A}_2$, then $\Pi(S)$ is obtained by gluing the spheres with boundary $\Pi'(A_i)$ and $\Pi'(A_j)$ along one common cuff, and is hence again a sphere with boundary. Third, if $S = A_i \in \mathcal{A}_1$, then $\Pi(S)$ is obtained from the sphere with boundary $\Pi'(A_i)$ by identifying pairs of points within one cuff, and is thus a projective plane with boundary.
    
    Notice that $\Pi(S^+)$ is obtained from $\Pi(S)$ by gluing a sphere with boundary along each cuff of $\Pi(S)$. The spheres with boundary that are glued are pairwise disjoint, since the distance of $A_i$ and $A_j$ is at least $4$ for $S=A_i\cup A_j\in\mathcal{A}_2$. Hence $\Pi(S)$ and $\Pi(S^+)$ are nested. In particular, $\Pi(S^+)$ is a sphere with boundary exactly when $\Pi(S)$ is and a projective plane with boundary exactly when $\Pi(S)$ is. Their respective number of cuffs can of course differ.
\end{proof}

\begin{lemma}\label{proj-planar-fw}
	For $S\in \mathcal{A}_1$ we have $\fw(G\cap\Pi(S))\ge r$.
    \end{lemma}
\begin{proof}
	Let $\Pi:=\Pi(S)$. 
    Let $\ell_i\in L_1$ be the leaf of $T$ which is contained in $S$. 
	By \Cref{set-props}, $\Pi$ is a projective plane with boundary. 
    Assume for a contradiction that $\fw(G\cap\Pi)<r$. 
    Let $o$ be a noose in $\widehat{\Pi}$ which intersects $G\cap\Pi$ in less than $r$ vertices. 
    The noose $o$ has to intersect $D_{i,0}$ since both $o$ and $D_{i,0}$ correspond to non-contractible curves in $\Pi$. 
    Since the noose intersects $G\cap\Pi$ in less than $r$ vertices, it is contained in $\Pi$ (recall that $d=3r+1$). 
    Since $o$ is one-sided, it is also a noose in $\surf$ intersecting $G$ in less than $r$ vertices. 
    This is a contradiction to the assumption that the face-width of $G$ is at least $\max\{w_{\ref{Yu}}(g, d), r+1\}$. 
\end{proof}

\begin{proof}[Proof of \Cref{thm:bounded_genus}]
    By the above construction and \Cref{set-props}, we obtain a collection of at most $2g$ pairs of nested spheres with boundary and projective planes with boundary $\left\{(\Pi(S),\Pi(S^+))\right\}_{S\in\mathcal{S}}$ forming a surface cover of $\surf$.
    By \Cref{3c-minor}, we obtain a minor $H_S$ for each $S\in\mathcal{S}$, such that $H_S$ is $3$-connected and agrees with $G$ in the neighborhood of $\Pi(S)$.
    Further, if $H_S$ is embedded in a projective plane, we know by \Cref{proj-planar-fw} that it has face-width at least $16$.
    Since the embeddings of $G$ and $H_S$ coincide within $\Pi(S)$, and $H_S$ is a supergraph of $G\cap\Pi(S)$, we conclude that $\fw(H_S) \ge \fw(G\cap\Pi(S)) \ge 16$ in this case.
    We denote by $R_S:=R\cap V(G\cap\Pi(S))$. 
    Then, by \Cref{thm:planar} and \Cref{proj-planar-case}, we obtain a face cover of size $O(t^4)$ for each $(H_S,R_S)$.
    We conclude by \Cref{lifting-covers} that $(G,R)$ admits a face cover of size $O(g \cdot t^4)$.
\end{proof}

We define the \emph{strong product} of two graphs $G$ and $H$ as $G\boxtimes H$, such that the vertex set corresponds to the cartesian product, that is, $V(G\boxtimes H):=V(G)\times V(H)$.
Further, we have an edge $(u,v)(u',v')\in E(G\boxtimes H)$ if either $u=u'$ and $vv'\in E(H)$, or $uu'\in E(G)$ and $v=v'$, or $uu'\in E(G)$ and $vv'\in E(H)$.

For $p \ge 3$ and $n = 2p$, we define the {\em bagel} as $B_n := K_2 \boxtimes C_p$, that is, the strong product of an edge and the cycle $C_p$. We denote the vertices of the first copy of $C_p$ by $v_1,\dots,v_p$, and the vertices of the second copy by $w_1,\dots,w_p$. Note that there is an edge between $v_i$ and $w_j$ if and only if $i-j\in\{0,\pm1\}$ (recall that indices are computed cyclically). To turn $B_n$ in a rooted graph, we mark each vertex as a root.

\begin{proposition} \label{prop:bagel}
For all even $n \ge 6$, $B_n$ has no (rooted) $K_{2,5}$ minor.
\end{proposition}

\begin{proof}
    Toward a contradiction, assume that $B_n$ has a $K_{2,5}$-model $M$ for some even $n = 2p \ge 6$, and assume that $n$ is minimum among such counterexamples. 
         
    As before, we denote by $M(x_1)$ and $M(x_2)$ the two branch sets that are adjacent to each \emph{central} branch set $M(y_1),\dots,M(y_5)$.
    Since every vertex of $B_n$ is a root, we can assume without loss of generality that each central branch set $M(y_i)$ contains a single vertex. Summing up, we have seven pairwise disjoint connected sets, namely $M(x_1)$, $M(x_2)$ plus five singleton sets $M(y_1)$, \ldots, $M(y_5)$, such that $M(y_i)\subseteq N(M(x_1))\cap N(M(x_2))$ holds for all $i \in [5]$. 
        
    In case there is some $i \in [2]$ and $j \in [p]$ such that $M(x_i)$ contains both $v_j$ and $w_j$, we can safely delete one from $M(x_i)$ without destroying our $K_{2,5}$-model.
    Observe further that for any $j \in [p]$, exchanging $v_j$ and $w_j$ 
    yields a graph isomorphism of $B_n$.
    Hence, we may assume without loss of generality that $M(x_1) \subseteq \{v_1,\ldots,v_p\}$ and $M(x_2) \subseteq \{w_1,\ldots,w_p\}$.
    Further, $M(x_1)$ and $M(x_2)$ induce connected subgraphs in their respective cycles, that is, each induces a path or a cycle.

    Since there is a total of $5$ vertices corresponding to the branch sets $M(y_i)$, we can assume without loss of generality that at least $3$ of them are contained in $\{v_1,\dots,v_p\}$.
    Thus in particular $M(x_1)$ cannot contain all vertices in $\{v_1,\dots,v_p\}$ and it induces a path $P=M(x_1)$ within $\{v_1,\dots,v_p\}$.
    Observe that for any such path $|N(P)\cap\{v_1,\dots,v_p\}|\le 2$, which contradicts the assumption that at least $3$ of the central branch sets are contained in $\{v_1,\dots,v_p\}$.
   \end{proof}

\section*{Acknowledgements}

The authors thank Gwena\"el Joret for discussions during the early stage of the research, Bojan Mohar for pointing out the reference~\cite{Yu97} to us, and Ken-Ichi Kawarabayashi for pointing out another relevant reference, which we ended up not using.

Samuel Fiorini, Abhinav Shantanam and Stefan Kober acknowledge funding from \emph{Fonds de la Recherche Scientifique} - FNRS through research project BD-DELTA (PDR 20222190, 2021--24). Samuel Fiorini was also funded by \emph{King Baudouin Foundation} through project BD-DELTA-2 (convention 2023-F2150080-233051, 2023--26). Michał Seweryn acknowledges funding from ERC-CZ project LL2328 of the Ministry of Education of Czech Republic. Yelena Yuditsky was supported by FNRS as a Postdoctoral Researcher. 

\bibliographystyle{alpha}
\bibliography{references}

\end{document}